\documentclass{amsart}

\usepackage{amsthm}
\usepackage{amsfonts}
\usepackage{amssymb}

\newcommand{\C}{{\mathbb{C}}}

\newcommand{\Inv}[1]{{\mathcal{I}}_{#1}}
\newcommand{\Rn}{\Delta_{\mathfrak{n}^-}}
\newcommand{\ideal}[1]{\big( {#1} \big)}

\newtheorem{theorem}{Theorem}[section]
\newtheorem{corollary}[theorem]{Corollary}
\newtheorem{lemma}[theorem]{Lemma}
\newtheorem{proposition}[theorem]{Proposition}

\newtheorem{definition}[theorem]{Definition}
\newtheorem{remark}[theorem]{Remark}

\begin{document}

\title{Divided difference operators for partial flag varieties}

\author{Julianna S. Tymoczko}

\address{Department of Mathematics, University of Iowa, Iowa City, IA 52242}

\email{tymoczko@math.uiowa.edu}

\thanks{\noindent
The author was partially supported by NSF grant 0402874.}

\begin{abstract}
Divided difference operators are degree-reducing operators on the cohomology of flag varieties that are used to compute algebraic invariants of the ring (for instance, structure constants).  We identify divided difference operators on the equivariant cohomology of $G/P$ for arbitrary partial flag varieties of arbitrary Lie type, and show how to use them in the ordinary cohomology of $G/P$.  

We provide three applications.  The first shows that all Schubert classes of partial flag varieties can be generated from a sequence of divided difference operators on the highest-degree Schubert class.  The second is a generalization of Billey's formula for the localizations of equivariant Schubert classes of flag varieties to arbitrary partial flag varieties.  The third gives a choice of Schubert polynomials for partial flag varieties as well as an explicit formula for each.  We focus on the example of maximal Grassmannians, including Grassmannians of $k$-planes in a complex $n$-dimensional vector space.\end{abstract}

\maketitle

\section{Introduction}
The divided difference operators of Bernstein-Gelfand-Gelfand and Demazure \cite{BGG}, \cite{D} are an essential tool in the Schubert calculus of flag varieties; a recent sample from diverse contexts includes \cite{Ma}, \cite{Pr}, \cite{LLM}, \cite{CLL}, \cite{Bi1}, \cite{KM}, \cite{F}.  This paper introduces nontrivial divided difference operators on the cohomology of partial flag varieties, including the classical Grassmannian of $k$-dimensional vector subspaces of $\mathbb{C}^n$.  The divided difference operators generate all Schubert classes from the top-degree Schubert class; they can be used to give a direct proof of a positive formula for localizations of equivariant Schubert classes; and the operators suggest a natural choice of Schubert polynomials, along with a closed formula for each polynomial.

The results in this paper apply to all partial flag varieties $G/P$, where $G$ is a semisimple complex linear algebraic group and $P$ is a parabolic subgroup containing a fixed Borel subgroup $B$.  For the sake of simplicity, this introduction presents the main results in the special case of Grassmannians of $k$-planes in $\C^n$.

The $i^{th}$ divided difference operator $\partial_i$ is a `discrete derivative' defined on the polynomial ring $\mathbb{Z}[x_1, \ldots, x_n]$ by 
\[\partial_i f = \frac{f(\cdots, x_i, x_{i+1}, \cdots) - f(\cdots, x_{i+1}, x_i, \cdots)}{x_i - x_{i+1}}.\]  
The kernel of each divided difference operator contains the ideal $I$ generated by nonconstant homogeneous symmetric polynomials.  Thus divided difference operators act on the ring $\mathbb{Z}[x_1, \ldots, x_n]/I$, which is isomorphic to the cohomology ring of the full flag variety by a classical result of Borel \cite{Bo}.  The torus-equivariant cohomology ring of the full flag variety can be presented in a similar way.  If $e_1(x), e_2(x), \ldots, e_n(x)$ denote the elementary symmetric functions in the variables $x_1, \ldots, x_n$, then the equivariant cohomology of the flag variety is isomorphic to $\mathbb{Z}[x_1, \ldots, x_n; y_1, \ldots, y_n]/J$ for the ideal $J$ generated by differences $e_i(x) - e_i(y)$.  
Divided difference operators exist on the equivariant cohomology ring, both in the $x$ variables and in the $y$ variables; however only the operators in the $x$ variables descend to ordinary cohomology, via the projection setting each $y_i$ to zero.  

The problem with extending this theory to Grassmannians is that the cohomology of Grassmannians is a submodule of the cohomology of flag variety.  Divided difference operators in the $x$ variables do not preserve this submodule.  Divided difference operators in the $y$ variables do, but they disappear in ordinary cohomology.   

We solve this problem using a different presentation of the equivariant cohomology ring due to Goresky-Kottwitz-MacPherson \cite{GKM} and often called GKM theory.  The GKM presentation describes equivariant classes in terms of their localizations, so an equivariant class is a tuple of polynomials satisfying certain linear compatibility conditions.  For example, let $G(k,n)$ be the Grassmannian of $k$-planes in $\C^n$.  The torus $T$ of $n \times n$ diagonal invertible matrices  acts on $\C^n$ by multiplication; this induces an action on $G(k,n)$.  Let ${\mathcal B}_k$ be the set of $n$-bit binary strings with exactly $k$ ones and denote elements of ${\mathcal B}_k$ by ${\bf b} = b_1 b_2 \cdots b_n$.   Then 
\[H^*_T(G(k,n); \C) \cong \left\{p: {\mathcal B}_k \rightarrow \C[t_1, \ldots, t_n] : \begin{array}{c} \textup{for each } {\bf b}, {\bf b}' \in {\mathcal B}_k  \textup{ such that}\\ b_i = b'_i \textup{ unless } i \in \{j_1, j_2\} \\ \textup{ the difference } p({\bf b}) - p({{\bf b}'}) \\ \textup{ is in the ideal } \big( t_{j_1} - t_{j_2} \big)  \end{array} \right\}.\]
(The variables $t_i$ are weights of the torus $T$; the earlier variables $x_i$ and $y_i$ are Chern classes of certain line bundles on the flag variety, and correspond to torus weights if appropriately reindexed.)  Section \ref{GKM section} treats GKM theory in detail.  

Our divided difference operators, denoted $\delta_i$, arise from a natural action of the permutation group $S_n$ on the equivariant cohomology of the Grassmannian obtained by simultaneously permuting both the entries of ${\mathcal B}_k$ and the variables $t_1, \ldots, t_n$.  This action is detailed in Section \ref{W action section} and has been studied in other contexts \cite[page 131]{A}, \cite{B}, \cite{T}.  Our construction of the divided difference operators on equivariant cohomology is combinatorial and explicit, and induces operators on the ordinary cohomology of the Grassmannian as well.  Our divided difference operators $\delta_i$ satisfy three key properties:
\begin{enumerate}
\item {\em Explicit formula (Proposition \ref{well-defined, nil coxeter}).} The map $\delta_i$ acts on $p \in H^*_T(G(k,n); \C)$ by 
\begin{equation}\label{familiar formula} 
\delta_i p = \frac{p - s_i \cdot p}{t_i - t_{i+1}}\end{equation}
where $s_i \cdot p$ denotes the action of $s_i = (i,i+1)$ on the equivariant class $p$;
\item {\em Compatibility with Schubert classes (Lemma \ref{formula for action}).}  The ring $H^*_T(G(k,n); \C)$ has a natural basis of equivariant Schubert classes $p_{{\bf b}}$ indexed by the elements ${\bf b} \in {\mathcal B}_k$.  The divided difference operators are compatible with equivariant Schubert classes, in the sense that
\begin{equation}\label{respects schubert classes}
\delta_i p_{{\bf b}} = \left\{ \begin{array}{ll} p_{s_i{\bf b}} & \textup{ if } b_i < b_{i+1} \textup{ as integers, and} \\ 0 & \textup{ otherwise;} \end{array} \right.\end{equation}
\item {\em NilCoxeter relations (Proposition \ref{braid relations}).}  The divided difference operators satisfy the nilCoxeter relations: first, if $w$ is a permutation with a minimal-length factorization $w= s_{i_1} s_{i_2} \cdots s_{i_k}$ then the divided difference operator $\delta_w$ is well-defined by $\delta_w =   \delta_{i_1} \delta_{i_2} \cdots \delta_{i_k}$; second, we have $\delta_i^2 = 0$ for each $i$.
\end{enumerate}

Formula \eqref{familiar formula} is similar to the classical divided difference operator but defines a very different morphism because it uses a different presentation of the cohomology ring.  The divided difference operators $\delta_i$ were observed for the flag variety by A.~Knutson in unpublished work \cite{K} and in greater generality by M.\ Brion \cite{B}, but both misidentified $\delta_i$ as the Bernstein-Gelfand-Gelfand/Demazure divided difference operators.  M.~Brion gives a geometric description of the operators $\delta_i$ in \cite{B} but while proving Properties (1)-(3) quotes proofs from \cite{D} which deal with the operators $\partial_i$.  In fact, the traditional divided difference operators $\partial_i$ of Bernstein-Gelfand-Gelfand \cite{BGG} and Demazure \cite{D} were extended to the localized presentation of the equivariant cohomology ring by B.~Kostant and S.~Kumar \cite{KK}.  Arabia proved that the Kostant-Kumar and Bernstein-Gelfand-Gelfand operators agree when viewed as algebra homomorphisms \cite[Theorem 3.3.1]{A}.  A more detailed exposition is in \cite{T2}.  During preparation of this manuscript we learned that the divided difference operators $\delta_i$ were defined by D.~Peterson in unpublished lecture notes \cite{P}.  Unfortunately, what is available in the literature from Peterson does not contain proofs of many key results.  Comments in the text highlight results which Peterson also proves, using different methods.  

We now describe three applications of the divided difference operator $\delta_i$.  The first, Theorem \ref{gen all classes}, states that all Schubert classes in the (equivariant) cohomology of the partial flag variety can be generated by successive applications of divided difference operators to the top-degree class.  (In the case of $G(k,n)$, the top-degree class is $p_{{\bf b}}$ for the bit-string ${\bf b}$ with $n-k$ zeroes followed by $k$ ones.)  This generalizes work of \cite{BGG} and \cite{D} and gives a new description of the ring $H^*_T(G(k,n); \C)$.  

The second is an extension to partial flag varieties of Billey's formula for flag varieties.  Billey's formula gives the localizations at each torus-fixed point of the equivariant Schubert classes \cite{Bi2}; in the case of $G(k,n)$, the formula gives each polynomial $p_{\bf b}({\bf b}')$.  The result is stated in Theorem \ref{Billey generalization}.  The proof given here applies to all partial flag varieties $G/P$ and so extends the beautiful determinantal formula for type $A_n$ given in \cite{LRS} and the combinatorial formulas for maximal Grassmannians in several classical types of \cite{Kr}, \cite{IN}.  We give a direct proof, unlike the non-constructive geometric proof in \cite[Corollary 11.3.14]{Ku}.  This section also demonstrates explicitly an isomorphism from a subring of the equivariant cohomology of flag varieties to that of Grassmannians.

In Section \ref{schubert poly section}, the third application proposes Schubert polynomials for all partial flag varieties that satisfy key conditions highlighted by Fomin and Kirillov \cite{FK}.  In the case of $G(k,n)$ these are: 
\begin{enumerate}
\item \label{min length perm defined} the Schubert polynomial $\mathfrak{S}_{{\bf b}}$ is a homogenous polynomial of degree equal to the minimal length of permutations $w$ satisfying $w ({\bf 1^k 0^{n-k}}) = {\bf b}$; 
\item the polynomial $\mathfrak{S}_{{\bf b}}$ is computed from the Grassmannian Schubert class $p_{{\bf b}}$; 
\item the polynomial $\mathfrak{S}_{{\bf b}}$ is positive in the simple roots $t_i - t_{i+1}$;
\item and the structure constants satisfy $\mathfrak{S}_{{\bf b}} \mathfrak{S}_{{\bf b'}} = \sum c_{\bf b b'}^{\bf c} \mathfrak{S}_{{\bf c}} \mod I$, namely they are the structure constants defined by the Schubert classes in the ordinary cohomology of the Grassmannian up to the ideal $I$.
\end{enumerate}
(Note that $I$ is independent of $k$.)  More unusually, we give a closed formula for calculating Schubert polynomials for partial flag varieties using Billey's formula.  If $w$ is the permutation satisfying Condition \eqref{min length perm defined} for ${\bf b}$, the Grassmannian Schubert polynomial $\mathfrak{S}_{{\bf b}}$ is defined to be the average of the localizations of the equivariant Schubert class $p_{w^{-1}}$ in the full flag variety (Theorem \ref{schubert poly theorem}).  In the flag variety, these Schubert polynomials are the Bernstein-Gelfand-Gelfand Schubert polynomials.  Unlike the double Schubert polynomials for Grassmannian permutations defined by Lascoux \cite{L}, our (single) Schubert polynomials can be computed from a closed formula and demonstrate invariance under the permutation subgroup $\langle s_i: i \neq k \rangle$.  Like Lascoux's, no divided difference operators act on our Grassmannian Schubert polynomials, though the $\delta_i$ act on the corresponding localized Schubert classes.

The author gratefully acknowledges helpful conversations with Anders Buch, Charles Cadman, Sergey Fomin, William Fulton, Mark Shimozono, and John Stembridge, as well as invaluable assistance from Megumi Harada.  
\tableofcontents
\subsection{Notation}
More details about the general theory may be found in classical references, e.g. \cite{H} or \cite{Sp}.  More details on the combinatorics of Weyl groups and roots may be found in \cite[Chapter 2]{BB}.

Let $G$ be a complex semisimple linear algebraic group and $P$ be a parabolic subgroup in standard position, namely containing a fixed Borel $B$ and maximal torus $T \subseteq B$.  If $g \in G$ then denote the corresponding flag in $G/B$ by $[g]$.  The Weyl group corresponding to $G$ is the quotient $N(T)/T$, where $N(T)$ denotes the normalizer of the torus.  Denote the Weyl group by $W$.  Let $W_J$ be the subgroup of $W$ defined by $P$.  This means there is a subset $J$ of the integers $\{1,2,\ldots, \textup{rk}(G)\}$ so that $W_J$ is the subgroup of $W$ generated by the corresponding simple reflections $W_J = \langle s_i : i \in J \rangle$.
The quotient $W/W_J$ indexes the Schubert cells of $G/P$ and, more generally, plays the same role for $G/P$ as $W$ does for $G/B$.

If $G=GL_n(\C)$ then the subgroup $B$ may be chosen to be the upper-triangular invertible matrices.  In this case $T$ is the subgroup of diagonal matrices, $W$ is the permutation group $S_n$, and the simple reflection $s_i$ is the transposition $(i,i+1)$.

Let $w \in W$ and choose a minimal-length factorization of $w$ into simple reflections, say $w = s_{i_1} \cdots s_{i_k}$.  Then the {\em length} of $w$ is $k$, and is denoted $\ell(w)$.   We denote the coset corresponding to $w \in W$ by $[w] \in W/W_J$. Each $[w] \in W/W_J$ has a unique minimal-length representative of $[w]$ in $W$, because $P$ is parabolic.  Suppose that $w'$ is this minimal-length representative.  The length of $[w]$ in $W/W_J$ is defined to be $\ell(w')$ and is denoted $\ell_P([w])$. The minimal-length representatives of the cosets in $W/W_J$ are often denoted $W^J$.  The Bruhat order on $W$ is defined by $w \geq v$ if and only if there is a reduced word $w = s_{i_1} \cdots s_{i_k}$ such that $v$ can be written as a substring $v = s_{i_{j_1}} \cdots s_{i_{j_l}}$.  The Bruhat order on $W/W_J$ is the ordinary Bruhat order on the minimal representatives, so that if $[w], [v] \in W/W_J$ are cosets with minimal-length representatives $w',v' \in W$ then $[w] \geq_P [v]$ if and only if $w' \geq v'$.

Denote the Lie algebras of $G$, $B$, $P$, and $T$ by $\mathfrak{g}$, $\mathfrak{b}$, $\mathfrak{p}$, and $\mathfrak{t}$ respectively.  The root system associated to $G$ is $\Delta$.  The roots are the eigenfunctions of the adjoint action of $T$ on $\mathfrak{g}$.  (If $t \in T$ and $X \in \mathfrak{g}$, then the adjoint action is given by $\textup{Ad}_t(X) = tXt^{-1}$.)  Given a root $\alpha$, the root space $\mathfrak{g}_{\alpha} \subseteq \mathfrak{g}$ is the eigenspace of $\alpha$.  The positive roots $\Delta^+$ are the roots with $\mathfrak{b} = \oplus_{\alpha \in \Delta^+} \mathfrak{g}_{\alpha}$.  The simple roots in $\Delta^+$ are denoted $R = \{\alpha_1, \alpha_2, \ldots, \alpha_n\}$, where $n$ is the rank of $G$.  The roots are partially ordered by the condition $\alpha > \beta$ if and only if $\alpha - \beta$ is a sum of positive roots.

In the case of $GL_n(\C)$, the Lie algebra $\mathfrak{g}$ consists of all $n \times n$ matrices.  The roots are the formal differences $t_i - t_j$ for $i,j \in \{1,\ldots,n\}$.  The positive roots are $t_i - t_j$ for $1 \leq i<j \leq n$ and the simple roots are $t_i-t_{i+1}$ for $i = 1, 2, \ldots, n-1$.

The subalgebra $\mathfrak{p}$ decomposes into a maximal semisimple part $\mathfrak{l}$ and a nilpotent part $\mathfrak{n}$ so that $\mathfrak{p} = \mathfrak{l} \oplus \mathfrak{n}$.  If $\mathfrak{g}_{\alpha}$ is a root subspace then either $\mathfrak{g}_{\alpha} \subseteq \mathfrak{p}$ or not.  Let $\mathfrak{n}^-$ be the span $\mathfrak{n}^- = \oplus_{\mathfrak{g}_{\alpha} \not \subseteq \mathfrak{p}} \mathfrak{g}_{\alpha}$.  By construction $\mathfrak{g} = \mathfrak{p} \oplus \mathfrak{n}^-$.

If $V$ is a $T$-stable vector subspace of $\mathfrak{g}$, denote the roots whose root subspaces are contained in $V$ by $\Delta_V$.  We will be concerned with $\Rn$, which is called $\Delta_{G,P}$ in \cite{GHZ}.  Observe that $\Rn = - \Delta_{\mathfrak{n}}$ since $\Delta_{\mathfrak{l}} = - \Delta_{\mathfrak{l}}$ and $\Delta = \Delta_{\mathfrak{p}} \sqcup \Rn$.

The following lemma is included for the sake of completeness.

\begin{lemma} \label{roots in n-}
\[\Rn = -\{\alpha: \alpha > \alpha_i \textup{ for at least one } i \not \in J\}.\]
\end{lemma}

\begin{proof}
The semisimple part $\mathfrak{l}$ is the Lie algebra generated by the simple root spaces $\{\mathfrak{g}_{\pm \alpha_i}: i \in J\}$ together with $\mathfrak{t}$.  The root system for $\mathfrak{l}$ is consequently 
\[ \Delta_{\mathfrak{l}} = \left\{\alpha \in \Delta: \alpha = \sum_{i \in J} c_i \alpha_i \textup{ for some } c_i \in \mathbb{Z}\right\}.\]
Since $\mathfrak{p} \supseteq \mathfrak{b}$ and $\mathfrak{p} \supseteq \mathfrak{l}$, we know that
\[\Delta_{\mathfrak{p}} = \Delta_{\mathfrak{l}} \cup \Delta^+.\]
Since $\Rn = \Delta - \Delta_{\mathfrak{p}}$, we conclude 
\[\Rn =  -\{\alpha: \alpha > \alpha_i \textup{ for at least one } i \not \in J\}.\]
\end{proof}

\section{GKM theory for $G/P$} \label{GKM section}
This section is a self-contained introduction to GKM theory, which is a combinatorial
construction of equivariant cohomology.  (Here, as in all parts of this paper, we use cohomology with complex coefficients.)

Let $X$ be a compact projective algebraic variety with an algebraic action of an $n$-dimensional torus $T$.   Suppose that $T$ acts on $X$ with finitely-many fixed points and finitely-many one-dimensional orbits.  In this case, if $O$ is a one-dimensional $T$-orbit in $X$, then the closure $\overline{O}$ is homeomorphic to $\C P^1$ and consists of the union of $O$ with two distinct $T$-fixed points in $X$ \cite[Section (7.1)]{GKM}.

If in addition $X$ satisfies a condition called {\em equivariant formality}, then $X$ is called a GKM space.  Equivariant formality refers to the degeneration of a certain spectral sequence; see \cite[Section (1.2)]{GKM} for a precise definition.  Every smooth complex algebraic variety is equivariantly formal, as is any variety with no odd-dimensional ordinary cohomology \cite[Theorem (14.1)]{GKM}.  In particular, each partial flag variety $G/P$ is equivariantly formal.  

Under these circumstances, the equivariant cohomology of $X$ is given by a combinatorial algorithm.  Denote the symmetric algebra in $\mathfrak{t}^*$ by $S(\mathfrak{t}^*)$.
\begin{proposition} (Goresky, Kottwitz, and MacPherson)
Let $v_1, \ldots, v_r$ be the $T$-fixed points of $X$.  If $O_{ij}$ is a one-dimensional $T$-orbit whose closure contains the two fixed points $v_i$ and $v_j$, write $\alpha_{ij}$ for the weight of the torus action on $O_{ij}$.  Then
\[H^*_T(X) \cong \left\{\begin{array}{ll} p: \{v_1, \ldots, v_r\} \rightarrow S(\mathfrak{t}^*): & \textup{ for each one-dimensional orbit } O_{ij}, \\ & \textup{ the difference } p(v_i) - p(v_j) \in \ideal{\alpha_{ij}} \end{array} \right\}.\] 
\end{proposition}

It is frequently useful to think of $p$ as a tuple of polynomials, namely an element of $S(\mathfrak{t}^*)^r$.  Then the ring structure of $H^*_T(X)$ is defined by coordinate-wise addition and multiplication.  When $X=G/B$, this formulation was discovered by \cite{KK} and \cite{A}.  The condition that the difference $p(v_i) - p(v_j) \in \ideal{\alpha_{ij}}$ is often called {\em the GKM condition}.  We will choose a basis for $\mathfrak{t}^*$ and identify $\C[t_1,\ldots,t_n]$ with $S(\mathfrak{t}^*)$ through the natural isomorphism.

The GKM conditions can be described combinatorially using the {\em moment graph} of the variety $X$.  The vertices of the moment graph are the $T$-fixed points of $X$.  There is an edge between two vertices in the graph if and only if there is a one-dimensional $T$-orbit whose closure contains the corresponding two fixed points.  Each edge is labeled by the $T$-weight on its corresponding one-dimensional orbit.  We assume that the graph has been directed by the choice of a suitably generic one-dimensional torus.  (This graph is the $1$-skeleton of $X$ under the image of the moment map.)

With these conventions, the previous theorem can be rewritten.  An equivariant class can be described as a copy of the moment graph whose vertices are labeled by elements of $S(\mathfrak{t}^*)$ such that the labels on each pair of vertices connected by an edge satisfy the GKM condition, namely $p(v_i) - p(v_j) \in \ideal{\alpha_{ij}}$.  Figure \ref{p2 classes} gives examples of equivariant classes for $\C \mathbb{P}^2$.
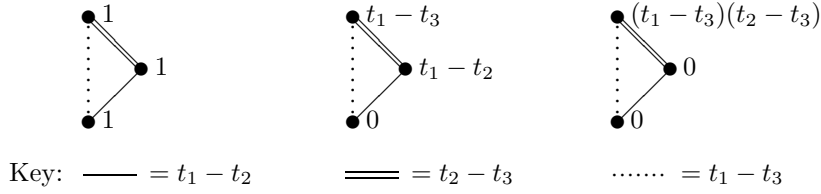
\begin{figure}[h]
\begin{picture}(350,63)(0,-40)
\put(50,-20){\circle*{5}}
\put(50,20){\circle*{5}}
\put(70,0){\circle*{5}}
\multiput(50,-20)(0,4){10}{\circle*{1}}
\put(50,-20){\line(1,1){20}}
\put(70,-1){\line(-1,1){20}}
\put(70,1){\line(-1,1){20}}

\put(55,-22){$1$}
\put(75,-2){$1$}
\put(55,18){$1$}

\put(150,-20){\circle*{5}}
\put(150,20){\circle*{5}}
\put(170,0){\circle*{5}}
\multiput(150,-20)(0,4){10}{\circle*{1}}
\put(150,-20){\line(1,1){20}}
\put(170,-1){\line(-1,1){20}}
\put(170,1){\line(-1,1){20}}

\put(155,-22){$0$}
\put(175,-2){$t_1-t_2$}
\put(155,18){$t_1-t_3$}

\put(250,-20){\circle*{5}}
\put(250,20){\circle*{5}}
\put(270,0){\circle*{5}}
\multiput(250,-20)(0,4){10}{\circle*{1}}
\put(250,-20){\line(1,1){20}}
\put(270,-1){\line(-1,1){20}}
\put(270,1){\line(-1,1){20}}

\put(255,-22){$0$}
\put(275,-2){$0$}
\put(255,18){$(t_1-t_3)(t_2-t_3)$}

\put(20,-42){Key:}
\put(48,-40){\line(1,0){20}}
\put(73,-42){$= \small{t_1 - t_2}$}
\put(147,-39){\line(1,0){20}}
\put(147,-41){\line(1,0){20}}
\put(172,-42){$= \small{t_2 - t_3}$}
\multiput(249,-40)(3,0){7}{\circle*{1}}
\put(274,-42){$= \small{t_1 - t_3}$}
\end{picture}
\caption{Examples of GKM classes for $\mathbb{CP}^2$} \label{p2 classes}
\end{figure}

The ordinary cohomology ring is the quotient of the equivariant cohomology ring by a maximal ideal in $S(\mathfrak{t}^*)$:
\[H^*(X) = \frac{H^*_T(X)}{\ideal{\alpha_1, \ldots, \alpha_n} S(\mathfrak{t}^*)}.\]

\subsection{The flow-up basis for $H^*_T(G/P)$} \label{flow-up basis subsection}

The vertices of the moment graph are partially ordered by the rule that $v \leq w$ if and only a sequence of directed edges leads from $w$ to $v$ in the graph.  We write $w \stackrel{\alpha}{\mapsto} v$ if there is an edge directed from $w$ to $v$ and labeled $\alpha$.

\begin{definition} A {\em flow-up class} $p_w \in H^*_T(X)$ corresponding to the
$T$-fixed point $w$ is a class that satisfies three conditions:
\begin{enumerate}
\item the class $p_w$ is supported exactly on the $T$-fixed points $v$ such that $v \geq w$;
\item the class $p_w$ is homogeneous; and
\item the localization $\displaystyle p_w(w) = \prod_{\footnotesize 
\begin{array}{c} \alpha \textup{ such that } \\w \stackrel{\alpha}{\mapsto} v \end{array}} \alpha$.
\end{enumerate}
\end{definition}

Each of the classes in Figure \ref{p2 classes} is in fact a flow-up class.  When flow-up classes exist for each $T$-fixed point in $X$, the set $\{p_w: w \in X^T\}$ forms a basis for $H^*_T(X)$, see \cite[Theorem 2.4.4]{GZ} or \cite[Proposition 2.13]{T}.  This is the case for a broad class of manifolds called Palais-Smale manifolds, which includes $G/P$.  In $H^*_T(G/P)$ the flow-up classes are equivariant Schubert classes, corresponding to the dual to the closure of $[Bw]$ in the sense of Graham \cite[Proposition 2.1]{G}.  (Note that this is not the Poincar\'{e} dual.)
The class $p_w$ was called $\xi^w$ in \cite{KK}.

Calculations in $H^*_T(X)$ are much easier if $H^*_T(X)$ has a basis of flow-up classes.  In a later section, we compute an explicit formula for the localizations $p_{[w]}([v])$ of flow-up classes in $H^*_T(G/P)$.

\section{The moment graph of $G/P$} \label{moment graph section}

In this section, we characterize the moment graph of $G/P$ with respect to the $T$-action in several ways.  The results here generalize Carrell's work \cite{C} on generalized flag varieties $G/B$ to all partial flag varieties $G/P$.  The moment graph is described in the following theorem, which is a variation of work of Guillemin-Holm-Zara \cite[Theorem 2.4]{GHZ}.
\begin{theorem} \label{moment graph description}
The moment graph of $G/P$ consists of the following data:
\begin{enumerate}
\item The vertices are the cosets $[w] \in W/W_J$.
\item There is an edge between two fixed points $[w] \neq [v]$ if and only if $[v] = [s_{\alpha}w]$ for some reflection $s_{\alpha} \in W$ with $\alpha > 0$ and $w^{-1}(\alpha) \in \pm \Rn$.  In this case the edge is labeled $\alpha$.
\item Suppose there is an edge between two vertices $[w] \neq [s_{\alpha}w]$ for $\alpha > 0$.   If $w^{-1}(\alpha)<0$ then the edge is directed $[w] \mapsto [s_{\alpha}w]$.  Otherwise $[s_{\alpha}w] \mapsto [w]$.
\end{enumerate}
\end{theorem}
Section \ref{cells and roots} describes the Schubert cells of $G/P$ in terms of certain subsets of roots and then uses this data to describe the $T$-fixed points and one-dimensional $T$-orbits in $G/P$.  In the algebraic context, this simplifies the proof of \cite[Theorem 2.4]{GHZ}.  We include our proof here because it can be used to construct explicitly the unique minimal-length representative for each coset in $W/W_J$.  We do this in Section \ref{min rep}.  If $w$ is a minimal-length representative for the coset $[w] \in W/W_J$, then we will see that the following are equal:
\begin{enumerate}
\item the length of $w$ in $W$;
\item the length of $[w]$ in $W/W_J$;
\item the dimension of the Schubert cell $[Bw]$ in $G/P$;
\item and the number of edges directed out of $[w]$ in the moment graph of $G/P$.
\end{enumerate}
Section \ref{sub and quotient} shows that the moment graph of $G/P$ can be realized combinatorially as the quotient of the moment graph of $G/B$ by the action of $W_J$.  Furthermore, the subgraph of covering relations of the moment graph of $G/P$ (namely, the edges between elements of $W/W_J$ that differ by length one) is the subgraph induced from the moment graph of $G/B$ by the minimal-length representatives of $W/W_J$.  This expands on the results of \cite[Proposition 1.1]{S} that show the quotient $W/W_J$ carries a poset structure induced by the Bruhat order on the minimal-length representatives.  The final part of this section gives several examples.

\subsection{Cells and roots} \label{cells and roots}

The next lemma is the essential step in this section.  It establishes a $T$-equivariant bijection between a certain subgroup of the unipotent group and the Schubert cell $[Bw]$ in $G/P$.  We will use this to explicitly identify the moment graph of $G/P$ as well as the minimal-length representatives for each $[w] \in W/W_J$.

In what follows, we will frequently treat $w \in W$ as if it were an element in $G$.  There are two facts we will use repeatedly:
\begin{enumerate}
\item The element $[wP] \in G/P$ is well-defined for all $w \in W$ because $T \subseteq P$.  Indeed, for all $g \in G$ the element $[gwP]$ is well-defined.
\item If $\mathfrak{g}_{\alpha}$ is a root subspace then $w \mathfrak{g}_{\alpha} w^{-1}$ is well-defined, and equals $\mathfrak{g}_{w(\alpha)}$.  In fact let $U_{\alpha}$ be the subgroup of $G$ with $\textup{Lie}(U_{\alpha}) = \mathfrak{g}_{\alpha}$.  Then $wU_{\alpha}w^{-1}$ is well defined and equals $U_{w(\alpha)}$.  In general if $u\in U_{\alpha}$ then $wuw^{-1}$ is defined up to conjugation by $T$.  (See Humphreys for more \cite[Theorem in Section 26.3]{H}.)  
\end{enumerate}

\begin{lemma} \label{special coset rep}
For each $[w] \in W/W_J$, let $U_w^P$ be the subgroup of the unipotent group in $B$ defined by the condition
\[\textup{Lie}(U_w^P) = \mathfrak{b} \cap \left( w\mathfrak{n}^-w^{-1}\right).\]
The map $U_w^P \rightarrow [Bw]$ given by $u \mapsto [uw]$ is a $T$-equivariant continuous bijection, where $T$ acts on $U_w^P$ by conjugation and on $G/P$ by (left) multiplication.
\end{lemma}

\begin{proof}  First we show that $U_w^P$ is a subgroup.  The unipotent subgroup $U$ can be factored into root subgroups $U = \prod_{\alpha \in \Delta^+} U_{\alpha}$ and the factorization is unique up to (arbitrary) fixed ordering of the roots \cite[Proposition 28.1]{H}.  The definition of $U_w^P$ implies that $U_w^P$ is a product of root subgroups $U_{\alpha}$, so we may write 
\begin{equation} \label{subgroup reps} U_w^P = \prod_{\footnotesize \begin{array}{c}
\alpha \in \Delta^+ s.t. \\ w^{-1}(\alpha) \in \Rn \end{array}} U_{\alpha}.\end{equation}
(Again, the ordering is fixed but arbitrary.)  The Chevalley commutator relation \cite[page 207]{Sp} states that if $u \in U_{\alpha}$ and $v \in U_{\beta}$ then
\[uvu^{-1}v^{-1} \in \prod_{\footnotesize \begin{array}{c} i,j \geq 1 \textup{ s.t.} \\ i\alpha + j \beta \in \Delta \end{array}} U_{i\alpha + j \beta}.\]  
Recall that $\mathfrak{n}^-$ is the direct sum of root spaces $\mathfrak{g}_{\alpha}$
such that $\alpha < -\alpha_i$ for at least one $i \not \in J$.  Hence if $w^{-1}(\alpha), w^{-1}(\beta) \in \Rn$ and $\alpha+\beta$ is a root then $w^{-1}(\alpha+\beta) \in \Rn$. The Chevalley commutator relation thus implies that $U_w^P$ is a subgroup.

Let $U'$ be the product of root subgroups $U_{\alpha}$ such that $w^{-1}(\alpha) \in \Delta_\mathfrak{p}$.  Order the roots $\alpha \in \Delta$ so that the unipotent subgroup factors $U = U_w^PU'$.  We will use this factorization to show that the map $U_w^P \rightarrow [Uw]$ is bijective.  (The Schubert cell $[Bw] = [Uw]$ because $B = UT$ and $w \in W$.)

First we show surjectivity.  Let $[uw]$ be any element of $[Uw] \subseteq G/P$ and write $u = u_1u_2$ where $u_1 \in U_w^P$ and $u_2 \in U'$.  Then $w^{-1}u_2w \subseteq P$ by definition of $U'$.  Hence $[uw]=[u_1w]$ has a representative in $[U_w^Pw]$.  

Now we show injectivity.  The representative $u_1$ is unique because if $[u_1w]=[u'w]$ for $u_1,u' \in U_w^P$ then $w^{-1}(u_1^{-1}u')w \subseteq P$.   The only element of $U_w^P$ satisfying this condition is $u_1^{-1}u'=id$.  

Finally we analyze the map $U_w^P \rightarrow [U_w^Pw]$.  It is $T$-equivariant by definition of the $T$-actions, and because $W$ is a quotient of the normalizer of $T$.  Let $g_w$ be any element of $G$ in the coset $w \in W$.  Then both the multiplication map $U_w^P \hookrightarrow U_w^Pg_w$ and the quotient map $U_w^Pg_w \rightarrow [U_w^Pw]$ are continuous.  The map $u \mapsto [uw]$ is their composition, and hence is continuous.    
\end{proof}

The next corollary proves that the moment graph of $G/P$ has vertices, edges, and labels on edges described in Theorem \ref{moment graph description}.

\begin{corollary}
The $T$-fixed points in $G/P$ are the cosets $[w] \in W/W_J$.  The $T$-stable curves in $G/P$ are
\[\{[U_{\alpha}w] \cup [s_{\alpha}w]: \textup{ for each } [w] \in W/W_J \textup{ and root subgroup } U_{\alpha} \subseteq U_w^P\}.\]
The weight of the $T$-action on the curve $[U_{\alpha}w] \cup [s_{\alpha}w]$ is $\alpha$.
\end{corollary}

\begin{proof}
Let $[w] \in W/W_J$.  The only $T$-fixed point under the adjoint action of $T$ on $U_w^P$ is the identity; it maps to $[w] \in G/P$ by the $T$-equivariant bijection of Lemma \ref{special coset rep}.  Similarly, the one-dimensional $T$-orbits of $U_w^P$ are exactly the non-identity elements of $U_{\alpha} \subseteq U_w^P$, so the one-dimensional orbits in the Schubert cell $[Uw]$ are exactly the sets $[U_{\alpha}w] - [w]$ for each root subgroup $U_{\alpha} \subseteq U_w^P$.  By inspection, the weight of the $T$-action on the orbit $[U_{\alpha}w] - [w]$ is $\alpha$.

The closure of each one-dimensional $T$-orbit in any GKM space (including $G/P$) consists of the $T$-orbit together with two extra points, each of which is fixed by the $T$-action \cite[Section (7.1)]{GKM}.  The set $[U_{\alpha}w] \subseteq [Uw]$ already contains one $T$-fixed point in $G/P$; we must find the other.  The closure of $[U_{\alpha}w]$ in $G/B$ is $[U_{\alpha}w] \cup [s_{\alpha}w]$ by \cite{C}.  The projection $G/B \rightarrow G/P$ is continuous so the closure of $[U_{\alpha}w]$ in $G/P$ contains $[s_{\alpha}w]$.  If $[s_{\alpha}w] \neq [w]$ in $W/W_J$ then we conclude that $[s_{\alpha}w]$ must be the second $T$-fixed point in the one-dimensional orbit in $G/P$.  By hypothesis the root $w^{-1}(\alpha) \in \Rn$ so $U_w^P \neq U_{s_{\alpha}w}^P$.  Using this hypothesis, Lemma \ref{U_w^P agrees on cosets} will show $[w] \neq [s_{\alpha}w]$.
\end{proof}

\subsection{The minimal word for $[w]$} \label{min rep}

We now describe explicitly the minimal-length word in $W$ for each $[w] \in W/W_J$.  Our main result is that the edges directed out of $[w]$ in the moment graph for $G/P$ are labeled exactly by roots whose corresponding root subgroups are contained in $U_w^P$, which in turn are exactly the ``inversions" of the minimal-length representative of $[w]$, as defined below.  This implies that the geometric quantity $\dim [Bw]$, the combinatorial quantity $\# \{\textup{edges directed out of } [w]\}$, and the algebraic quantity $\ell_P([w])$ are all equal.  We also include here several technical combinatorial results about roots and Weyl group elements.  We begin with the following definition.

\begin{definition}
The $P$-inversions at $w$ are the roots in the set
 \[\Inv{w}^P = \{\beta \in \Delta^+: w^{-1}(\beta) \in \Rn \}.\]  
 \end{definition}

(A similar set is often called $N(w)$, e.g. \cite[page 102]{BB}.)  If $P=B$ or equivalently $\Rn = \Delta^-$ we write $\Inv{w}$ instead of $\Inv{w}^B$.  Lemma \ref{special coset rep} showed that $\Inv{w}^P$ consists of the labels on the edges directed out of $w$ in the moment graph for $G/P$.  The next lemma follows from the definitions. 

\begin{lemma}  \label{U_w^P agrees on cosets} If $[w]=[v]$ in $G/P$ then $\Inv{w}^P=\Inv{v}^P$. \end{lemma}  

\begin{proof} Recall that if $s_i$ is a simple reflection then $s_i(\alpha) = \alpha + c_i \alpha_i$ for some $c_i \in \mathbb{Z}$.  Hence if $u \in W_J$ then $u^{-1}(\alpha) = \alpha + \sum_{i \in J} c_i \alpha_i$.  Lemma \ref{roots in n-} implies that $u^{-1}w^{-1}(\beta) \in \Rn$ if and only if $w^{-1}(\beta) \in \Rn$.
\end{proof}

The previous lemma implies that $\Inv{[w]}^P$ is well-defined on cosets $[w] \in W/W_J$, by setting $\Inv{[w]}^P = \Inv{w}^P$ for any (and hence every) representative $w \in [w]$.

\begin{lemma}
For each $w \in W$, there is a unique element $v \in W$ with $\Inv{v} = \Inv{w}^P$.\end{lemma}

\begin{proof}
We use Kostant's result \cite{Ko} that the set $\{\Inv{v}: v \in W\}$ is in bijection with the collection of subsets $S$ of $\Delta^+$ that satisfy the two conditions
\begin{enumerate}
\item if $\alpha, \beta \in S$ and $\alpha + \beta \in \Delta^+$ then $\alpha + \beta \in S$; and
\item if $\alpha, \beta \in \Delta^+ - S$ and $\alpha + \beta \in \Delta^+$ then $\alpha + \beta \not \in S$. 
\end{enumerate}
We also use the characterization of $\Rn$ as the set of roots $\alpha$ with $\alpha < -\alpha_i$ for at least one $i \not \in J$, as in Lemma \ref{roots in n-}.

Observe that if $\alpha, \beta \in \Inv{w}^P$ and $\alpha+\beta \in \Delta$ then $w^{-1}(\alpha) + w^{-1}(\beta) < -\alpha_i$ for at least one $i \not \in J$, for instance the $i$ such that $w^{-1}(\alpha) < -\alpha_i$.  Thus $\alpha + \beta \in \Inv{w}^P$.

Suppose $\alpha, \beta \in \Delta^+ - \Inv{w}^P$ and $\alpha + \beta \in \Delta$.  Write $w^{-1}(\alpha) = \sum c_i \alpha_i$ and $w^{-1}(\beta) = \sum d_i \alpha_i$.  We describe the $c_i$ (respectively $d_i$): 
\begin{itemize}
\item If $\alpha \in \Inv{w} - \Inv{w}^P$ then the $c_i$ are nonpositive, and nonzero only when $i \in J$.
\item If $\alpha \in \Delta^+ - \Inv{w}$ then the $c_i$ are positive.
\end{itemize}
Consequently the coefficients in the sum $\sum (c_i+d_i) \alpha_i$ can be negative only for $i \in J$.  Thus $w^{-1}(\alpha+\beta) \in \Delta_{\mathfrak{p}}$ and so $\alpha+\beta \not \in \Inv{w}^P$.

The set $\Inv{w}^P$ satisfies conditions (1) and (2) so there is a unique Weyl group element $v$ with $\Inv{v} = \Inv{w}^P$.
\end{proof}

\begin{corollary}
For each $w \in W$, the Weyl group element $v$ defined by $\Inv{v} = \Inv{w}^P$ is the unique minimal-length element of the coset $[w] \in W / W_J$.
\end{corollary}

\begin{proof}
It is well-known that for each $w \in W$ the length $\ell(w) = \# \Inv{w}$ (see e.g. \cite[Proposition 4.4.4]{BB}).  If $u \in [w]$ then $\Inv{u} \supseteq \Inv{u}^P$ with equality if and only if $u = v$.  This proves the claim.
\end{proof}

It follows that the length $\ell_P([w])$ can also be characterized as the number of edges directed out of $[w]$ in the moment graph of $G/P$.  If $[w]$ and $[s_{\alpha}w]$ are two fixed points joined by an edge then $\ell_P([w]) \neq \ell_P([s_{\alpha}w])$.  (These are combinatorial implications of the geometric statement that $G/P$ is Palais-Smale.)

%
\begin{lemma} \label{graph automorphisms}
The action of the Weyl group $W$ on $W/W_J$ given by left-multiplication induces a collection of graph automorphisms on the moment graph of $G/P$.
\end{lemma}

\begin{proof}
Left multiplication by elements of the group $W$ permutes the cosets of $W/W_J$, namely the vertices of the moment graph.  We need only confirm that the edges of the moment graph are preserved by this action.  Inside the Weyl group 
\[ws_{\alpha}v = (ws_{\alpha}w^{-1})wv= s_{w(\alpha)}wv.\]
Note that $[s_{w(\alpha)}wv] = [wv]$ if and only if $v^{-1}s_{\alpha}v \in W_J$, equivalently $[s_{\alpha}v]=[v]$.  This means that there is an edge between two distinct cosets $[s_{\alpha}v]$ and $[v]$ in the moment graph for $G/P$  if and only if there is an edge between $[ws_{\alpha}v]= [s_{w(\alpha)}wv]$ and $[wv]$ in the moment graph for $G/P$.  
\end{proof}

The next lemma characterizes how $[s_iw]$ and $[w]$ differ.  It is a partial converse to a result of Deodhar \cite{D}.  Unlike in $G/B$, it is possible for $[s_iw] = [w]$ in $G/P$.

\begin{lemma}  \label{inversions}
Let $w \in W$ and let $s_i$ be a simple reflection.
\begin{enumerate}
\item If $[s_iw]=[w]$ then $s_i \Inv{w}^P = \Inv{w}^P$.  
\item If $[s_iw]>_P[w]$ then $s_i \Inv{s_iw}^P =\Inv{w}^P \cup \{-\alpha_i\}$.
\end{enumerate}
\end{lemma}

\begin{proof}
From Lemma \ref{graph automorphisms}, we know that multiplication by each simple reflection $s_i$ induces a graph automorphism of the moment graph of $G/P$.  Suppose $[v] \mapsto [s_{\alpha}v]$ is an edge with $\alpha \neq \alpha_i$.  Then this automorphism sends the edge $[v] \mapsto [s_{\alpha}v]$ to the edge $[s_iv] \mapsto [s_is_{\alpha}v] = [s_{s_i(\alpha)} (s_iv)]$ because
\begin{itemize}
\item  $v^{-1}(\alpha)$ and $(s_iv)^{-1}(s_i(\alpha))$ have the same sign, and 
\item $s_i(\alpha)$ is positive as long as $\alpha \neq \alpha_i$.  
\end{itemize}

If $[s_iw]=[w]$ then $\Inv{s_iw}^P =  \Inv{w}^P$ by Lemma \ref{U_w^P agrees on cosets}.  Hence $s_i$ permutes the edges directed out of $[w]$, namely $s_i\Inv{w}^P=\Inv{w}^P$.

If $[s_iw]>_P [w]$ then $\alpha_i$ labels an edge directed from $[s_iw]$ to $[w]$ in the moment graph of $G/P$.  Hence $\alpha_i \in \Inv{s_iw}^P$ and $\alpha_i \not \in \Inv{w}^P$.  Let $\beta \neq \alpha_i$ be a positive root.  Note that $s_i(\beta) \in \Inv{w}^P$ if and only if $s_i(\beta)$ is positive and $\beta \in \Inv{s_iw}^P$. 
The simple reflection $s_i$ permutes the positive roots not equal to $\alpha_i$.  We conclude $\Inv{s_iw}^P = s_i \Inv{w}^P \cup \{\alpha_i\}$. 
\end{proof}

\begin{corollary} \label{step by one}
If $s_i$ is a simple reflection in $W$ and $[w]$ is a coset in $W/W_J$
then $|\ell_P([s_iw]) - \ell_P([w])| \leq 1$.  
\end{corollary}

\begin{corollary} \label{simple edges}
If $w \in W^J$ and $s_i$ is an arbitrary simple reflection then either $[s_iw]=[w]$ in $W/W_J$ or $s_iw \in W^J$.
\end{corollary}

\begin{proof}
We know that $\Inv{w}^P = \Inv{w}$ since $w$ is the minimal-length representative of $[w]$.

If $[s_iw]>_P[w]$ then $\alpha_i \not \in \Inv{w}^P$ and $\alpha_i \in \Inv{s_iw}^P$.  Hence $\alpha_i \in \Inv{s_iw}$ so $s_iw>w$.  It follows that $\ell(s_iw) = \ell(w)+1$.  Since $\ell_P([s_iw])=\ell_P([w])+1$ and $\ell_P([w])=\ell(w)$, we conclude $\ell(s_iw)=\ell_P([s_iw])$.  This means that $\Inv{s_iw}=\Inv{s_iw}^P$ and $s_iw \in W^J$.

Now suppose $[s_iw]<_P[w]$.  Let $u$ be the minimal-length representative of $[s_iw]$.
Since $[s_iu]=[w]$, the previous case proved $s_iu$ is a minimal-length representative
of $[w]$.  Each coset of $W/W_J$ has a unique minimal-length element so $s_iu=w$.
\end{proof}

\subsection{The moment graph of $G/P$ as a quotient graph and as a subgraph} \label{sub and quotient}

The Bruhat order is traditionally described by the graph of its covering relations, namely the set of Weyl group elements together with edges between $v$ and $w$ if $v>w$ and $v$ and $w$ differ by length one.  The moment graph of $G/B$ contains the Bruhat order as a subgraph but also includes additional edges corresponding to reflections that change length by more than one.  The next result generalizes this statement to $G/P$.  It is paired with a dual result.  The net outcome is that the moment graph of $G/P$ can be constructed combinatorially from that of $G/B$ either as a subgraph (with additional edges) or as a quotient graph.  The quotient $W/W_J$ is endowed with a poset structure inherited from the Bruhat order on the minimal-length representatives of $W/W_J$.  This is proven in \cite[Proposition 1.1]{S}; we include a proof here because of its brevity.  

Recall that if $(V,E)$ is a graph and $V_1, V_2, \ldots, V_k$ is a partition of $V$ into $k$ disjoint parts, then the {\em quotient graph} has vertex set $\{1,2,\ldots,k\}$ and has an edge $(i,j)$ if and only if there is an edge $(u,v) \in E$ for some $u \in V_i$ and $v \in V_j$ \cite[59.4.2]{MS}.

\begin{theorem}
\begin{enumerate}
\item The moment graph of $G/P$ is the quotient graph of the moment graph of $G/B$ determined by the partition of $W$ into cosets $W/W_J$.  
\item The subgraph of the moment graph of $G/B$ induced by the Weyl group elements $W^J$ is a subgraph of the moment graph of $G/P$ that contains all the covering relations in the moment graph of $G/P$.
\end{enumerate}
\end{theorem}

\begin{proof}
By Theorem \ref{moment graph description}, the vertices of the moment graph of $G/P$ are the cosets $W/W_J$.  There is an edge between two vertices $[w], [v] \in W/W_J$ if and only if $[v] = [s_{\alpha}w]$ for some $\alpha$; equivalently, if and only if there is an edge between two representatives of the cosets in the moment graph of $G/B$.  By definition, the moment graph of $G/P$ is the quotient of the moment graph of $G/B$ by $W_J$.

We now prove the second part.  Suppose $w$ and $s_{\alpha}w$ are both in $W^J$ and that there is an edge between $w$ and $s_{\alpha}w$ in the moment graph for $G/B$.  Then $[w] \neq [s_{\alpha}w]$ because $W^J$ is a set of distinct coset representatives for $W/W_J$, so there is an edge between $[w]$ and $[s_{\alpha}w]$ in the moment graph of $G/P$.  

Conversely, suppose $w, v \in W^J$ differ in length by one, and that there is an edge from $[w]$ to $[v]$ in the moment graph of $G/P$, say labeled $\alpha$.  This means $\alpha \in \Inv{w}^P$.  Since $w \in W^J$ we know that $\Inv{w} = \Inv{w}^P$ and so $s_{\alpha}w<w$ in the Bruhat order.  Thus $\ell(s_{\alpha}w) \leq \ell(v)$.  It follows that $s_{\alpha}w = v$ since $[s_{\alpha}w] = [v]$ and $v$ has minimum length in the coset $[v]$.  Hence the edge from $w$ to $v$ is in the moment graph of $G/B$.
\end{proof}

\begin{remark}\label{remark: covering relations determine graph} The covering relations determine the entire moment graph of $G/P$.  In fact, any edge $[w] \stackrel{\alpha}{\mapsto} [s_{\alpha}w]$ is determined by the edges $[v] \mapsto [s_iv]$ for simple reflections $s_i$ together with transitivity of left multiplication.  To see this, denote an undirected edge in the moment graph by $[w] \leftrightarrow [v]$, including the self-edges like $[s_iw]=[w]$.  If $s_{\alpha}$ is a reflection and $s_{\alpha} = s_{i_1}s_{i_2} \cdots$ is any minimal-length factorization, then the (undirected) path $[w] \leftrightarrow [s_{i_1}w] \leftrightarrow [s_{i_2}s_{i_1}w] \leftrightarrow \cdots \leftrightarrow [s_{\alpha}w]$ determines the edge $[w] \leftrightarrow [s_{\alpha}w]$ in the moment graph of $G/P$.
\end{remark}

For instance, Figure \ref{quotient/sub} shows the moment graph for the flag variety $GL_3/B$, namely full flags in $\C^3$, together with the moment graph for the Grassmannian of lines in $\C^3$, namely $\mathbb{P}^1$. 
\begin{figure}[h]
\begin{picture}(360,75)(10,-45)
\put(180,-30){\circle*{5}}
\put(180,30){\circle*{5}}
\put(165,-15){\circle*{5}}
\put(195,-15){\circle*{5}}
\put(165,15){\circle*{5}}
\put(195,15){\circle*{5}}

\put(180,-30){\line(-1,1){15}}
\put(195,15){\line(-1,1){15}}
\put(195,-15){\line(-1,1){30}}

\put(179,-29){\line(1,1){15}}
\put(164,16){\line(1,1){15}}
\put(164,-14){\line(1,1){30}}
\put(181,-31){\line(1,1){15}}
\put(166,14){\line(1,1){15}}
\put(166,-16){\line(1,1){30}}

\multiput(180,-30)(0,4){15}{\line(0,1){2}}
\multiput(165,-15)(0,4){7}{\line(0,1){2}}
\multiput(195,-15)(0,4){7}{\line(0,1){2}}

\multiput(180,-30)(3,3){5}{\circle*{2}}
\multiput(165,-15)(0,4){7}{\circle*{2}}
\multiput(195,15)(-3,3){5}{\circle*{2}}

\put(280,-30){\circle*{5}}
\put(265,-15){\circle*{5}}
\put(295,15){\circle*{5}}

\put(280,-30){\line(-1,1){15}}
\put(264,-14){\line(1,1){30}}
\put(266,-16){\line(1,1){30}}
\multiput(280,-30)(2,6){8}{\circle*{1}}
\multiput(281,-27)(2,6){7}{\circle*{1}}

\put(20,20){\line(1,0){20}}
\put(20,1){\line(1,0){20}}
\put(20,-1){\line(1,0){20}}
\multiput(20,-20)(4,0){5}{\line(1,0){2}}
\multiput(20,-40)(3,0){7}{\circle*{2}}

\put(43,18){$=t_1-t_2$}
\put(43,-2){$=t_2-t_3$}
\put(43,-22){$=t_1-t_3$}
\put(43,-42){$=$ quotient}

\put(185,-32){$e$}
\put(199,-17){$s_2$}
\put(199,13){$s_2s_1$}
\put(183,33){$s_2s_1s_2$}
\put(150,-17){$s_1$}
\put(140,13){$s_1s_2$}

\put(285,-32){$[e]$}
\put(246,-17){$[s_1]$}
\put(299,13){$[s_2s_1]$}

\put(160,-45){$GL_3/B$}
\put(260,-45){$G(1,3)$}
\end{picture}
\caption{The moment graph for $GL_3/B$ and for $G(1,3)$}\label{quotient/sub}
\end{figure}
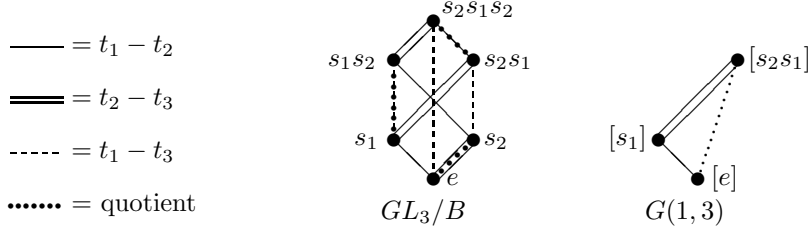
 If the graph for $G(1,3)$ is constructed as a quotient graph, then the edges marked by large dots are collapsed in the moment graph of $GL_3/B$.  Note that the edges between cosets are well-defined.  If the graph for $G(1,3)$ is constructed as a subgraph, then it is the graph induced by the vertices $e$, $s_1$, and $s_2s_1$ from the moment graph of $GL_3/B$.  Any vertex that is not on an edge $t_i-t_{i+1}$ can be considered to have a self-edge labeled $t_i-t_{i+1}$.  We can then reconstruct the rest of the edges by transitivity.

\subsection{Examples}\label{examples} We describe moment graphs for the Grassmannian of $k$-planes in $\C^n$ as well as the isotropic Grassmannians of classical types.  

It is common to denote the vertices of the moment graph for $G(k,n)$ as $n$-bit binary strings with exactly $k$ ones \cite{KT}.  The permutation $w \in S_n$ acts on the $n$-bit binary string $b_1 \cdots b_n$ by $w (b_1 \cdots b_n) = b_{w(1)} \cdots b_{w(n)}$. If ${\bf b}$ and ${\bf c}$ are distinct bit strings with $(ij) {\bf b} = {\bf c}$ then the vertices ${\bf b}$ and ${\bf c}$ are joined by an edge labeled $t_i - t_j$.  If $j>i$ the edge is directed from the bit string whose $j^{th}$ bit is one to the bit string whose $j^{th}$ bit is zero, for instance $1010 \mapsto 1100$ in $G(2,4)$.  To recover the permutation corresponding to a bit string ${\bf b}$, choose a minimal-length directed path from ${\bf b}$ to $1^k 0^{n-k}$ whose edges are labeled by simple roots $t_i -t_{i+1}$, for instance
\[{\bf b} \stackrel{t_{i_m} - t_{i_m+1}}{\mapsto} (i_m, i_m+1) {\bf b} \stackrel{t_{i_{m-1}}-t_{i_{m-1}+1}}{\mapsto}  \cdots \stackrel{t_{i_1}-t_{i_1+1}}{\mapsto} 1^k0^{n-k}.\]  
The minimal-length permutation for ${\bf b}$ is $(i_m, i_m+1) (i_{m-1}, i_{m-1}+1) \cdots (i_1, i_1+1)$.  Figure \ref{quotient/sub} gave the moment graph of $G(1,3) = \mathbb{P}^1$.  The moment graph for $G(2,4)$ is an octahedron and can be found in \cite[Figure 7]{KT} or in \cite[Figure 11]{T3}. 

The general (maximal) Grassmannian is the collection of $k$-planes in $\C^N$ that are isotropic with respect to some linear form, namely planes $V$ that satisfy the condition $\langle v_1, v_2 \rangle = 0$ for all $v_1, v_2 \in V$.  Their moment graphs are similar to that of the Grassmannian $G(k,n)$.  We describe each Lie type in turn.

First we describe the isotropic $k$-planes in $\C^{2n}$ that satisfy a skew-symmetric form, namely maximal Grassmannians of type $C_n$.  Define a skew-symmetric form on the basis vectors $e_i$ by 
\begin{itemize}
\item $\langle e_i, e_{2n+1-i} \rangle = 1$ if $i<2n+1-i$ and $\langle e_i, e_{2n+1-i} \rangle = -1$ if $i>2n+1-i$; 
\item $\langle e_i, e_j \rangle = 0$ otherwise.  
\end{itemize}
The vertices of the moment graph for the corresponding isotropic Grassmannian are the $2n$-bit binary strings $b_1 \cdots b_{2n}$ with exactly $k$ ones such that $b_i b_{2n+1-i} = 0$ for all $i$.  The whole permutation group $S_n$ does not act on this set of bit strings, but the subgroup generated by the following permutations does:
\begin{description}
\item[Type I] $(ij)(2n+1-i, 2n+1-j)$ for all $i < j \leq n$,
\item[Type II] $(i,2n+1-j)(j,2n+1-i)$ for all $i < j \leq n$, and 
\item[Type III] $(i,2n+1-i)$ for all $i \leq n$.
\end{description}
Let $s$ be one of the permutations listed above.  If ${\bf b}$ and ${\bf c}$ are two different bit strings such that $s {\bf b} = {\bf c}$ then ${\bf b}$ and ${\bf c}$ are joined by an edge.  This edge is labeled $t_i - t_j$ if $s$ is Type I, $t_i + t_j$ if $s$ is Type II, and $2t_i$ if $s$ is Type III.  The directions are induced from the directions in the moment graph of the ordinary Grassmannian, in the following sense.  If $s$ is one of the permutations listed above then $s {\bf b} \geq b$ if the relevant condition below holds. 
\begin{description}
\item[$s$ is Type I] $b_i \leq b_j$ and $b_{2n+1-j} \leq b_{2n+1-i}$; 
\item[$s$ is Type II] $b_i \leq b_{2n+1-j}$ and $b_j \leq b_{2n+1-i}$; and 
\item[$s$ is Type III] $b_i \leq b_{2n+1-i}$.
\end{description}
If $s {\bf b} \neq {\bf b}$ then at least one of the inequalities in this list will be a strict equality, and the corresponding edge will be directed $s{\bf b} \mapsto {\bf b}$. The simple reflections for type $C_n$ are Type I reflections of the form 
\[s_i = (i,i+1)(2n-i,2n+1-i) \textup{  for  } i<n\] 
and the Type III reflection $s_n= (n,n+1)$.  The minimum-length permutation associated to the fixed point ${\bf b}$ is found in the same way as for $G(k,n)$, namely if
\[{\bf b} {\mapsto} s_{i_m} {\bf b} {\mapsto} s_{i_{m-1}} s_{i_m} {\bf b} {\mapsto} \cdots {\mapsto}  s_{i_1} \cdots s_{i_{m-1}} s_{i_m} {\bf b}  = 1^k0^{n-k}\]  
is a minimal-length path from ${\bf b}$ to $1^k0^{n-k}$ labeled by simple roots then the minimum-length permutation for ${\bf b}$ is $s_{i_m} s_{i_{m-1}} \cdots s_{i_1}$.  Figure \ref{Grass examples} shows the subgraph of the moment graph of $IG(2,6)$ induced from the edges labeled $t_1-t_2$, $t_2-t_3$, and $2t_3$, namely those corresponding to simple reflections. 
 \begin{figure}[h]
\begin{picture}(360,60)(-50,-30)
\put(60,0){\circle*{5}}
\put(90,0){\circle*{5}}
\put(120,20){\circle*{5}}
\put(120,-20){\circle*{5}}
\put(150,20){\circle*{5}}
\put(150,-20){\circle*{5}}
\put(180,20){\circle*{5}}
\put(180,-20){\circle*{5}}
\put(210,20){\circle*{5}}
\put(210,-20){\circle*{5}}
\put(240,0){\circle*{5}}
\put(270,0){\circle*{5}}

\put(60,1){\line(1,0){30}}
\put(60,-1){\line(1,0){30}}
\put(90,0){\line(3,2){30}}
\multiput(90,0)(3,-2){10}{\circle*{1}}
\multiput(120,20)(2,0){15}{\circle*{1}}
\put(120,-20){\line(3,4){30}}
\put(120,-21){\line(1,0){30}}
\put(120,-19){\line(1,0){30}}
\put(150,21){\line(1,0){30}}
\put(150,19){\line(1,0){30}}
\put(150,-20){\line(1,0){30}}
\multiput(180,20)(2,0){15}{\circle*{1}}
\put(180,20){\line(3,-4){30}}
\put(180,-21){\line(1,0){30}}
\put(180,-19){\line(1,0){30}}
\put(240,0){\line(-3,2){30}}
\multiput(240,0)(-3,-2){10}{\circle*{1}}
\put(240,1){\line(1,0){30}}
\put(240,-1){\line(1,0){30}}

\put(25,-3){\small 110000}
\put(275,-3){\small 000011}

\put(95,-30){\small 100100}
\put(210,-30){\small 001001}
\put(170,-30){\small 010001}
\put(135,-30){\small 100010}

\put(95,25){\small 011000}
\put(210,25){\small 000110}
\put(170,25){\small 001010}
\put(135,25){\small 010100}

\put(230,7){\small 000101}
\put(67,-10){\small 101000}

\put(-40,20){\line(1,0){15}}
\put(-40,1){\line(1,0){15}}
\put(-40,-1){\line(1,0){15}}
\multiput(-40,-20)(2,0){7}{\circle*{1}}
\put(-23,18){\small $=t_1-t_2$}
\put(-23,-2){\small $=t_2-t_3$}
\put(-23,-22){$=2t_3$}
\end{picture}
\caption{The moment graph of $IG(2,6)$ of type $C_3$}\label{Grass examples}
\end{figure}
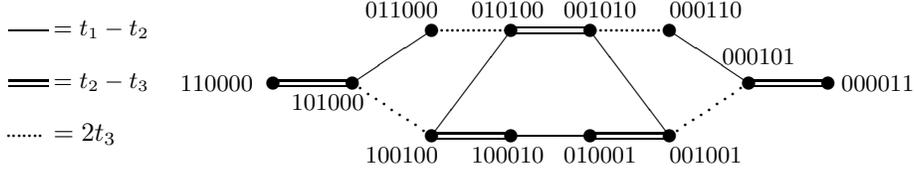
The other edges have been omitted to make the Figure more readable.  They can be reconstructed by the procedure described in Remark \ref{remark: covering relations determine graph}; for instance the full moment graph has an edge $010001 \mapsto 010100$ labeled $t_1 - t_3$.  (The figure has also been rotated so that edges are directed from right to left rather than up to down.)

The moment graphs for Grassmannians of isotropic $k$-planes in $\C^N$ that satisfy a symmetric bilinear form, namely maximal Grassmannians of type $B_n$ and $D_n$, are similar.  Both use the symmetric bilinear form defined on the basis vectors $e_i$ by $\langle e_i, e_{N+1-i} \rangle = 1$ and $\langle e_i, e_j \rangle = 0$ otherwise.  In each case, we will describe the fixed points, the simple reflections, and the directions of the edges corresponding to simple reflections in the moment graph.  The edges corresponding to simple reflections will determine the rest of the edges as described in Remark \ref{remark: covering relations determine graph}.  (The simple reflections generate all the reflections inductively according to the rule that if $s_{\alpha}$ and $s_{\beta}$ are two reflections then $s_{\beta} s_{\alpha} s_{\beta}$ is also a reflection, corresponding to the root $s_{\beta}(\alpha)$.) We will describe how simple reflections act on roots; this action extends naturally to all reflections.  The minimal-length Weyl group element corresponding to each fixed point can be reconstructed analogously to types $A_n$ and $C_n$.

In type $B_n$, we consider the Grassmannian of isotropic $k$-planes in $\C^{2n+1}$.  The fixed points are indexed by bit strings of length $2n+1$ with exactly $k$ ones such that $b_i b_{2n+2-i} = 0$ for each $i$.  (In particular, this implies that $b_{n+1}=0$.)  The group that acts on these fixed points is generated by the simple reflections
\begin{itemize}
\item $s_i = (i,i+1)(2n+1-i, 2n+2-i)$ for all $i=1, \ldots, n-1$ and
\item $s_n = (n,n+2)$.
\end{itemize}
When $i<n$ the edge between ${\bf b}$ and $s_i {\bf b}$ is labeled $t_i - t_{i+1}$.  When $i=n$ the edge between ${\bf b}$ and $s_i {\bf b}$ is labeled $t_n$.  The edge is directed $s_i {\bf b} \mapsto {\bf b}$ if
\begin{itemize}
\item $b_i \leq b_{i+1}$ and $b_{2n+1-i} \leq b_{2n+2-i}$ for $i=1, \ldots, n-1$ (with at least one strict inequality), or
\item $b_n < b_{n+2}$ for $i=n$.
\end{itemize}
If $i \neq n$ the simple reflection $s_i$ exchanges the variables $t_i$ and $t_{i+1}$ and fixes the other variables; the simple reflection $s_n$ sends $t_n$ to $-t_n$ and fixes the other variables.  An arbitrary reflection $s_{\alpha}$ factors into a product of simple reflections, extending this action from simple reflections to all reflections; Remark \ref{remark: covering relations determine graph} now determines the moment graph.

In type $D_n$, we consider the Grassmannian of isotropic $k$-planes in $\C^{2n}$.  The fixed points are indexed by bit strings of length $2n$ with exactly $k$ ones for $k = 1, \ldots, n-2$ or $k=n$, that also satisfy $b_ib_{2n+1-i}=0$ (as in type $C_n$).  The group that acts on these fixed points is generated by the simple reflections
\begin{itemize}
\item $s_i = (i,i+1)(2n-i,2n+1-i)$ for $i=1,\ldots,n-1$ and
\item $s_n = (n,n+2)(n-1,n+1)$.
\end{itemize}
When $i<n$ the edge between ${\bf b}$ and $s_i {\bf b}$ is labeled $t_i - t_{i+1}$.  When $i=n$ the edge between ${\bf b}$ and $s_i {\bf b}$ is labeled $t_{n-1}+t_n$.  The edge is directed $s_i {\bf b} \mapsto {\bf b}$ if
\begin{itemize}
\item $b_i \leq b_{i+1}$ and $b_{2n-i} \leq b_{2n+1-i}$ for $i=1, \ldots, n-1$ (with at least one strict inequality), and
\item $b_n \leq b_{n+2}$ and $b_{n-1} \leq b_{n+1}$ for $i=n$ (with at least one strict inequality).
\end{itemize}
If $i \neq n$ the simple reflection $s_i$ exchanges the variables $t_i$ and $t_{i+1}$ and fixes the other variables; the simple reflection $s_n$ sends $t_n$ to $-t_n$ and $t_{n-1}$ to $-t_{n-1}$, and fixes the other variables. Remark \ref{remark: covering relations determine graph} now determines the moment graph.

\section{A $W$-action on $G/P$} \label{W action section}

In this section we define an action of the Weyl group $W$ on $H^*_T(G/P)$ and give an explicit formula for the action of simple reflections on the basis of equivariant Schubert classes.

Recall that $W$ acts on $\mathfrak{t}^*$ by reflections, with the action of $w \in W$ on the root $\alpha$ denoted $w(\alpha)$.  This action extends naturally to a $W$-action on $S(\mathfrak{t}^*)$; each $w \in W$ can be considered as a $\C$-algebra homomorphism of $S(\mathfrak{t}^*)$.

\begin{lemma}
For each $w \in W$ and $p  : W/W_J \rightarrow S(\alpha_1, \ldots, \alpha_n)$, the element $w \cdot p : W/W_J \rightarrow  S(\alpha_1, \ldots, \alpha_n)$ is defined by the condition
\[ \left( w \cdot p\right)([v]) = w p(w^{-1}[v]).\] 
This makes each $w \in W$ a well-defined twisted $S(\mathfrak{t}^*)$-module homomorphism on $H^*_T(G/P)$, namely if $c \in S(\mathfrak{t}^*)$ and $p \in H^*_T(G/P)$ then $w \cdot (cp) = (w(c)) ( w \cdot p)$.
\end{lemma}

\begin{proof}
The group action of $W$ on functions $W/W_J \rightarrow S(\alpha_1, \ldots, \alpha_n)$ is well-defined because 
\[(w_1 \cdot (w_2 \cdot p))([v]) = 
w_1(w_2 \cdot p)(w_1^{-1}[v]) = w_1w_2p(w_2^{-1}w_1^{-1}[v]) = \left((w_1w_2) \cdot p \right)([v]).\]

We confirm that $W$ acts on $H^*_T(G/P)$, namely that if $p$ is a class in $H^*_T(G/P)$ then $w \cdot p$ is also in $H^*_T(G/P)$.  To do this, we check that 
\[(w \cdot p)([s_{\alpha}v]) - (w \cdot p)([v]) \in \ideal{\alpha}\] 
for all vertices $[v]$ and edges $[s_{\alpha}v] \mapsto [v]$. The identity 
\[ w^{-1}[s_{\alpha}v] = [(w^{-1}s_{\alpha}w)w^{-1}v] = [s_{w^{-1}(\alpha)}w^{-1}v]\]  
implies that
\[(w \cdot p)([s_{\alpha}v])- (w \cdot p)([v]) = w p ([s_{w^{-1}(\alpha)}w^{-1}v]) - wp([w^{-1}v]).\] 
Lemma \ref{graph automorphisms} proved that each $w$ is a graph automorphism on the moment graph of $G/P$.  Since $[s_{\alpha}v] \mapsto [v]$ is an edge in the moment graph of $G/P$, we conclude that $[s_{w^{-1}(\alpha)}w^{-1}v] \mapsto [w^{-1}v]$ is also an edge, and is labeled $w^{-1}(\alpha)$.  Thus 
\[wp ([s_{w^{-1}(\alpha)}w^{-1}v]) - wp([w^{-1}v]) \in \ideal{\alpha},\]  
namely $w \cdot p \in H^*_T(G/P)$.

Finally, we show each $w \in W$ acts as a twisted $S(\mathfrak{t}^*)$-module homomorphism.  Addition in $H^*_T(G/P)$ is defined coordinate-wise, so for each $w \in W$ and $p_1, p_2 \in H^*_T(G/P)$ we have $w \cdot (p_1 + p_2) = w \cdot p_1 + w \cdot p_2$.  If $p \in H^*_T(G/P)$ and $c \in S(\mathfrak{t}^*)$ then for each $[v] \in W/W_J$ we have
\[ (w \cdot (cp))[v] = w(cp)([w^{-1}v]) = w(c) ((w \cdot p)[v]).\]
So the action of $w \in W$ twists the $S(\mathfrak{t}^*)$-action on $H^*_T(G/P)$. \end{proof}

This action can be described explicitly for simple reflections $s_i$.  If we consider $p \in H^*_T(G/P)$ as a tuple of polynomials, one polynomial for each vertex of the moment graph, then the action of $s_i$ permutes the indices in each polynomial by $s_i$ and exchanges polynomials on either side of an edge labeled $\alpha_i$.   Figure \ref{action example} demonstrates this, using the same notation as Figure \ref{p2 classes}.  
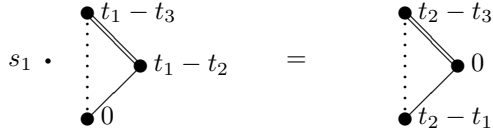
\begin{figure}[h]
\begin{picture}(350,40)(0,-20)
\put(70,0){$s_1$}
\put(85,2){\circle*{2}}
\put(175,0){$=$}

\put(100,-20){\circle*{5}}
\put(100,20){\circle*{5}}
\put(120,0){\circle*{5}}
\multiput(100,-20)(0,4){10}{\circle*{1}}
\put(100,-20){\line(1,1){20}}
\put(120,-1){\line(-1,1){20}}
\put(120,1){\line(-1,1){20}}

\put(105,-22){$0$}
\put(125,-2){$t_1-t_2$}
\put(105,18){$t_1-t_3$}

\put(220,-20){\circle*{5}}
\put(220,20){\circle*{5}}
\put(240,0){\circle*{5}}
\multiput(220,-20)(0,4){10}{\circle*{1}}
\put(220,-20){\line(1,1){20}}
\put(240,-1){\line(-1,1){20}}
\put(240,1){\line(-1,1){20}}

\put(225,-22){$t_2 - t_1$}
\put(245,-2){$0$}
\put(225,18){$t_2-t_3$}
\end{picture}
\caption{The action of $s_1$ on a Schubert class in $H^*_T(\mathbb{P}^2)$} \label{action example}
\end{figure}

\subsection{A formula for the action of simple reflections}
We give an explicit formula for the action of a simple reflection on the equivariant Schubert basis
of $H^*_T(G/P)$.  This formula is the basis for our divided difference operators.  

We begin with a short corollary of the results on $P$-inversions, included here for completeness.  (If we use specific representatives $w \in W^J$, this corollary follows from results on $G/B$  \cite[Proposition 4.7]{T}.)  Recall that Lemma \ref{U_w^P agrees on cosets} implies that for each $[w] \in W/W_J$, the set $\Inv{[w]}^P$ is well-defined by the rule $\Inv{[w]}^P = \Inv{w}^P$.

\begin{corollary} \label{first entry}
If $p_{[w]}$ is a flow-up class  in $H^*_T(G/P)$ corresponding to $[w] \in W/W_J$ and $s_i$ is a simple reflection such that $[s_iw] \mapsto [w]$ then
\[p_{[w]}([s_iw]) = \prod_{\alpha \in \Inv{[w]}^P} s_i(\alpha)= s_ip_{[w]}([w]).\]
\end{corollary}

\begin{proof}
Let 
\[q =  \prod_{\alpha \in \Inv{[s_iw]}^P, \alpha \neq \alpha_i} \alpha.\]
Since $\ell_P([s_iw]) = \ell_P([w])+1$ and $p_{[w]}$ is a flow-up class, every vertex $[v] \neq [w]$ with an edge $[s_iw] \mapsto [v]$ satisfies $p_{[w]}([v])=0$.  By comparing degrees, we conclude that $p_{[w]}([s_iw])$ is a scalar multiple $cq$ for some $c \in \C$.  

Lemma \ref{inversions} shows that $s_iq = p_{[w]}([w])$ and that neither $q$ nor $p_{[w]}([w])$ is divisible by $\alpha_i$.  It follows that $q - p_{[w]}([w]) \in \ideal{\alpha_i}$ and that $c=1$ is the the only scalar satisfying $cq - p_{[w]}([w]) \in \ideal{\alpha_i}$.  The GKM conditions imply $p_{[w]}([s_iw])=q$.
\end{proof}

Next we generalize to $G/P$ a classical result for $G/B$ \cite[Lemma 2.4]{BGG}.

\begin{lemma} \label{diamonds}
Suppose that $[s_iw]>_P[w]$.  If there is an edge $[s_{\alpha}w] \mapsto [w]$ with $\ell_P([s_{\alpha}w])= \ell_P([w])+1$  in the moment graph of $G/P$ then the configuration of Figure \ref{diamond picture} is in the moment graph.
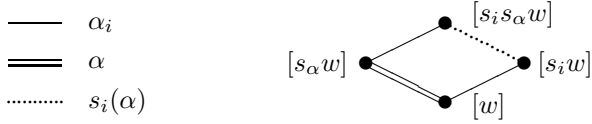
\begin{figure}[h]
\begin{picture}(350,40)(0,-20)
\put(175,15){\circle*{5}}
\put(145,0){\circle*{5}}
\put(205,0){\circle*{5}}
\put(175,-15){\circle*{5}}

\put(175,15){\line(-2,-1){30}}
\multiput(175,15)(2,-1){15}{\circle*{1}}
\put(175,-16){\line(-2,1){30}}
\put(175,-14){\line(-2,1){30}}
\put(175,-15){\line(2,1){30}}

\put(185,-19){$[w]$}
\put(210,-3){$[s_iw]$}
\put(185,16){$[s_is_{\alpha}w]$}
\put(115,-3){$[s_{\alpha}w]$}

\put(10,15){\line(1,0){20}}
\put(10,1){\line(1,0){20}}
\put(10,-1){\line(1,0){20}}
\multiput(10,-15)(2,0){11}{\circle*{1}}

\put(40,13){$\alpha_i$}
\put(40,-2){$\alpha$}
\put(40,-17){$s_i(\alpha)$}
\end{picture}
\caption{Schematic of diamond shape (angles not to scale)} \label{diamond picture}
\end{figure}
\end{lemma}

\begin{proof}
There is an edge $[s_{\alpha}w] \mapsto [w]$ so the root $\alpha$ is in $\Inv{[s_{\alpha}w]}^P$.  Lemma \ref{inversions} shows that $s_i \left( \Inv{[s_{\alpha}w]}^P - \{\alpha_i\}\right) \subseteq \Inv{[s_is_{\alpha}w]}^P$ regardless of whether $[s_is_{\alpha}w] = [s_{\alpha}w]$ or, if not, the direction of the edge between them.  In particular $s_i(\alpha) \in \Inv{[s_is_{\alpha}w]}^P$. The reflection corresponding to $s_i(\alpha)$ is $s_is_{\alpha}s_i$.  Since $(s_is_{\alpha}s_i)[s_i s_{\alpha}w] = [s_iw]$, we conclude $[s_is_{\alpha}w] \mapsto [s_iw]$.  In particular $\ell_P([s_is_{\alpha}w]) \geq \ell_P([s_iw]) + 1$.  Lemma \ref{step by one} says $\ell_P([s_is_{\alpha}w]) \leq \ell_P([s_{\alpha}w])+1$.  Since $\ell_P([s_{\alpha}w]) = \ell_P([s_iw]) = \ell_P([w])+1$, we see that $\ell_P([s_is_{\alpha}w]) = \ell_P([w])+2$ and the moment graph is as shown in Figure \ref{diamond picture}.
\end{proof}

We will now prove the main result of this section, an explicit formula for $s_i \cdot p_{[w]}$.

\begin{lemma} \label{formula for action}
Let $p_{[w]}$ be a flow-up class for $[w]$ in $H^*_T(G/P)$ and let $s_i$ be a simple reflection.  Then
\[s_i \cdot p_{[w]} = \left\{ \begin{array}{ll} p_{[w]} & \textup{ if } [s_iw] \geq [w] \textup{ and } \\
	p_{[w]}-\alpha_ip_{[s_iw]} & \textup{ if } [s_iw]<_P [w]. \end{array} \right.\]
\end{lemma}

\begin{proof}  We decompose $s_i \cdot p_{[w]} = \sum c_{[v]}p_{[v]}$ in terms of the basis of flow-up classes.  If $\ell_P([v])> \ell_P([w])$ then $c_{[v]} = 0$ because $s_i \cdot p_{[w]}$ is a homogeneous class of degree $\ell_P([w])$ and the coefficients $c_{[v]}$ have degree at least zero.  

Let $[u]$ be a minimal element of $W/W_J$ with $c_{[u]} \neq 0$.  We will evaluate the expression $s_i \cdot p_{[w]} = \sum c_{[v]}p_{[v]}$ at $[u]$.  Each flow-up class $p_{[v]}$ is supported only on $[u]$ with $[u] \geq_P [v]$, so 
\begin{equation} \label{eqn: minimal coeff}
(s_i \cdot p_{[w]})([u]) = \sum c_{[v]}p_{[v]}([u]) = c_{[u]} p_{[u]}([u])
\end{equation}
by the minimality hypothesis on $[u]$.  The reflection $s_i$ changes the length of a fixed point by at most one.  If $\ell_P([v]) < \ell_P([w])$ then  $\ell_P([s_iv]) \leq \ell_P([w])$ and in particular $[s_i v ] \not \geq [w]$ unless $[s_i v] = [w]$.  Hence $(s_i \cdot p_{[w]})([u]) = s_i p_{[w]}([s_iu]) = 0$ unless $u = s_iw$.
We conclude that $c_{[v]}=0$ if $\ell_P([v]) < \ell_P([w])$ except perhaps when both $[u]=[s_iw]$ and $[u] < [w]$.

Thus the only possible nonzero $c_{[v]}$ are $[v]$ with $\ell_P([v]) = \ell_P([w])$ and, if $[s_iw]<[w]$, the coefficient $c_{[s_iw]}$.

Consider the case when $[s_iw] \geq [w]$ first.  Let $[u]$ be a minimal element of $W/W_J$ with $c_{[u]} \neq 0$.  Equation \eqref{eqn: minimal coeff} implies that $(s_i \cdot p_{[w]})([u]) \neq 0$, and in particular the moment graph has an edge $[s_iu] \stackrel{\alpha}{\mapsto} [w]$ for some $\alpha$.  If $[u] \neq [w]$ then $\alpha \neq \alpha_i$.  Then $\alpha \in \Inv{[s_iu]}^P$ implies that $s_i \alpha \in \Inv{[u]}^P$ by Lemma \ref{inversions}.  The reflection $s_{s_i(\alpha)} = s_i s_{\alpha} s_i$ so 
\[[s_{s_i(\alpha)}u] = [s_is_{\alpha}s_iu] = [s_iw].\]
Hence the moment graph has an edge $[u] \mapsto [s_iw]$ and so $\ell_P([u]) > \ell_P([s_iw])$, which by hypothesis is at least $\ell_P([w])$.  This contradiction shows that $[u] = [w]$ is the only minimal element of $W/W_J$ with $c_{[u]} \neq 0$, and so also the only element of length $\ell_P([w])$ with $c_{[u]} \neq 0$.  Moreover Equation \eqref{eqn: minimal coeff} implies that
\[c_{[w]} = \frac{(s_i \cdot p_{[w]})([w])}{p_{[w]}([w])} = \frac{s_i p_{[w]}([s_iw])}{p_{[w]}([w])}.\]
If $[s_iw] \mapsto [w]$ then Corollary \ref{first entry} shows $c_{[w]} = 1$.  If $[s_iw] = [w]$ then Lemma \ref{inversions} shows $c_{[w]} = 1$. 

Now suppose $[s_iw]<_P[w]$.  All fixed points $[v]$ with $[v] < [s_iw]$ have $c_{[v]}=0$ so Equation \eqref{eqn: minimal coeff} applies to $u=s_iw$.  Together with Corollary \ref{first entry}, it implies that
\[c_{[s_iw]}=\frac{s_ip_{[w]}([w])}{p_{[s_iw]}([s_iw])} = -\alpha_i.\]
We will now identify $c_{[u]}$ for fixed points $[u]$ of length $\ell_P([w])$ by evaluating 
\[s_i \cdot p_{[w]} = -\alpha_ip_{[s_iw]} + \sum c_{[v]}p_{[v]}\] 
at various fixed points.  Since the sum is taken over $[v]$ with $\ell_P([v]) = \ell_P([w])$, the equation obtained by evaluation at $[u]$ is
\begin{equation} \label{eqn: other eval}
(s_i \cdot p_{[w]})([u]) = -\alpha_ip_{[s_iw]}([u])+ c_{[u]}p_{[u]}([u]).
\end{equation}  

We begin with $[u] = [w]$.  Corollary \ref{first entry} says $-\alpha_i p_{[s_iw]}([w]) = -\alpha_i s_i p_{[s_iw]}([s_iw])$, which is $-p_{[w]}([w])$ by Lemma \ref{inversions}.  Since $(s_i \cdot p_{[w]})([w]) = 0$ we obtain $c_{[w]} = 1$.

Next let $[u] \neq [w]$ be a fixed point with $\ell_P([u]) = \ell_P([w])$.  Equation \eqref{eqn: other eval} shows that $c_{[u]} =0$ unless $p_{[s_iw]}([u]) \neq 0$ or $p_{[w]}([s_iu]) \neq 0$.  The former implies that there is an edge $[u] \mapsto [s_iw]$ while the latter implies that there is an edge $[s_iu] \mapsto [w]$ by length considerations and the definition of flow-up classes.  Lemma \ref{diamonds} shows that if $[u] \mapsto [s_iw]$ then there is an edge $[s_iu] \mapsto [w]$.  Conversely, if $[s_iu] \mapsto [w]$ is an edge, say labeled $\alpha$, then $[u] = [(s_i s_{\alpha}s_i)s_iw]$ and so $[u] \mapsto [s_iw]$ is an edge labeled $s_i(\alpha)$.  Hence $c_{[u]} = 0$ unless the following picture is in the moment graph:
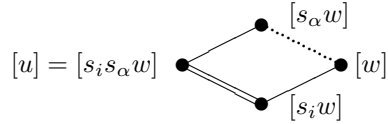
\begin{picture}(350,50)(0,-25)
\put(175,15){\circle*{5}}
\put(145,0){\circle*{5}}
\put(205,0){\circle*{5}}
\put(175,-15){\circle*{5}}

\put(175,15){\line(-2,-1){30}}
\multiput(175,15)(2,-1){15}{\circle*{1}}
\put(175,-16){\line(-2,1){30}}
\put(175,-14){\line(-2,1){30}}
\put(175,-15){\line(2,1){30}}

\put(185,-19){$[s_iw]$}
\put(210,-3){$[w]$}
\put(185,16){$[s_{\alpha}w]$}
\put(80,-3){$[u]=[s_is_{\alpha}w]$}
\end{picture}

Lemma \ref{inversions} showed that if the downward pointing edges at $[s_iw]$ and $[u]$ are labeled by $\Inv{[s_iw]}^P$ and $\Inv{[u]}^P$ respectively, then the downward pointing edges at $[w]$ and $[s_iu]$ are labeled $s_i\Inv{[s_iw]}^P \cup \{\alpha_i\}$ and $s_i \Inv{[u]}^P \cup \{\alpha_i\}$, respectively.
 Let 
\[q = \prod_{\scriptsize \begin{array}{c} \alpha \in \Inv{[u]}^P s.t. \\
\hspace{0em} [s_{\alpha}u] \neq [s_iw] \end{array}} \alpha.\]
Each edge $[u] \mapsto [s_{\alpha}u] \neq [s_iw]$ points to a vertex with $p_{[s_iw]}([s_{\alpha}u])=0$ by definition of flow-up classes.  Thus the localization $p_{[s_iw]}([u])$ is in the ideal $\ideal{q}$.  A similar argument shows that $p_{[w]}([s_iu]) \in \ideal{(\alpha_i)(s_iq)}$. 

Comparing degrees, we see that $p_{[s_iw]}([v])$ is in fact a scalar multiple of $q$, say $cq$ for $c \in \C$.  The scalar $c$ is determined by the GKM condition $cq - p_{[s_iw]}([s_iw]) \in \ideal{s_i(\alpha)}$ because $s_i(\alpha)$ is relatively prime to $q$.  (Indeed, the edges at a fixed vertex in the moment graph of a GKM space are labeled by positive roots, none of which are multiples or repetitions of each other.)
The same scalar $c$ satisfies the equation $c \alpha_i(s_iq) -  \alpha_is_ip_{[s_iw]}([s_iw]) \in \ideal{\alpha}$, which is the GKM condition for the edge $[s_iu] \mapsto [w]$ because $p_{[w]}([w]) = \alpha_is_ip_{[s_iw]}([s_iw])$.  No other scalar $c$ satisfies the equation because $\alpha$ is relatively prime to $\alpha_i(s_iq)$.  So $p_{[w]}([s_iu]) = \alpha_i s_i p_{[s_iw]}([u])$.  In other words $(s_i \cdot p_{[w]})([u]) = -\alpha_ip_{[s_iw]}([u])$.   Equation \eqref{eqn: other eval} implies that $c_{[u]}=0$. 
\end{proof}

We reformulate the previous result for later use.

\begin{corollary}\label{basic formula}
Let $p_{[w]} \in H^*_T(G/P)$ be the flow-up class corresponding to $[w] \in W/W_J$ and let $s_i$ be a simple reflection.  Then for each $[v] \in W/W_J$
\[p_{[w]}([v]) = \left\{ \begin{array}{ll} s_i(p_{[w]}([s_iv])) & \textup{ if } [s_iw] \geq [w] \textup{ and}\\
		s_i(p_{[w]}([s_iv])) + \alpha_ip_{[s_iw]}([v]) & \textup{ if } [s_iw] < [w].\end{array} \right.\]
\end{corollary}

\section{A divided difference operator} \label{props of div diff section}

Bernstein-Gelfand-Gelfand and Demazure simultaneously defined divided difference operators in order to determine the relationship between two presentations of the cohomology ring of the flag variety: the Borel presentation (as a quotient of a polynomial ring by an ideal of symmetric polynomials), and the geometric presentation (in terms of the basis of Schubert classes) \cite{BGG}, \cite{De}.  We refer to their operators as BGG divided difference operators.  In this section we define a new divided difference operator on both $H^*_T(G/P)$ and on $H^*(G/P)$, for arbitrary complex semisimple linear algebraic groups $G$ and arbitrary parabolics $P$.  

We begin with a review of the key aspects of BGG divided difference operators, from our perspective.  The BGG divided difference operator is a degree-lowering rational operator on $S(\mathfrak{t}^*)$.  The formula for the action of the $i^{th}$ BGG divided difference operator $\partial_i$ on $f \in S(\mathfrak{t}^*)$ is
\[ \partial_i f = \frac{f - s_if}{\alpha_i}.\]
We will use two properties of the BGG divided difference operator.  First:
\begin{itemize}
\item The BGG divided difference operator satisfies the two nil-Coxeter relations. First, if $s_{i_1} \cdots s_{i_k}$ and $s_{j_1} \cdots s_{j_k}$ are two reduced factorizations for the same word, then $\partial_{i_1} \cdots \partial_{i_k} = \partial_{j_1} \cdots \partial_{j_k}$.  Second, $\partial_i^2 = 0$ for each $i$.
\end{itemize}
The second property describes how divided difference operators act on Schubert polynomials.  A Schubert polynomial corresponding to $w\in W$ is a representative $\mathfrak{S}_w \in S(\mathfrak{t}^*)$ of the Schubert class in $H^*(G/P)$ corresponding to $w$.  There is an extensive literature on how best to choose the representative, e.g. \cite{LS}, \cite{FK}.  For the purposes of this paper, any convention can be chosen as long as the following are satisfied:  
\begin{itemize}
\item If $\partial_i$ is a BGG divided difference operator and $\mathfrak{S}_w$ is a Schubert polynomial corresponding to $w\in W$ then
\[\partial_i (\mathfrak{S}_w) = \left\{ \begin{array}{ll} \mathfrak{S}_{ws_i} & \textup{ if } ws_i<w \textup{ and} \\ 0 & \textup{ otherwise.} \end{array} \right.\]
\item The Schubert polynomial $\mathfrak{S}_e$ is a positive integer.
\end{itemize}
The Schubert polynomials defined in \cite{BGG} satisfy these properties, as do others.

BGG operators satisfy a Leibniz formula, proven for type $A_n$ in several places \cite[Equation 2.2]{M} and generalized in the following lemma.

\begin{lemma} \label{Macdonald lemma}
For each $f,g \in S(\mathfrak{t}^*)$ and BGG divided difference operator $\partial_i$
\[\partial_i(fg) = \partial_i(f) g +(s_if) \partial_i(g).\]
\end{lemma}

\begin{proof}
Expand:
\[\begin{array}{ll} \partial_i(fg) &= \displaystyle \frac{fg-(s_if)(s_ig)}{\alpha_i} \\
&= \displaystyle \frac{fg - (s_if)g}{\alpha_i} + \frac{(s_if)g-(s_if)(s_ig)}{\alpha_i} = \partial_i(f)g+(s_if)\partial_i(g).\end{array}\]
\end{proof}

We now define a second collection of divided difference operators $\delta_i$ on the GKM presentation of the $S(\mathfrak{t}^*)$-module $H^*_T(G/P)$.  We use $\delta_i$ to distinguish the new divided difference operators from the BGG divided difference operators $\partial_i$, which will also be used in what follows.  At the end of this section, we discuss results for ordinary cohomology, as well as how $\partial_i$ and $\delta_i$ are related.  Mark Shimozono helpfully pointed out that the following result was also given by D.\ Peterson \cite{P}.

\begin{proposition} \label{well-defined, nil coxeter}
For each simple reflection $s_i \in W$ define a divided difference operator $\delta_i$ that acts
on the class $p \in H^*_T(G/P)$ as a $\C$-module homomorphism by
\[ \delta_i p = \frac{p-s_i \cdot p}{\alpha_i}\]
Each $\delta_i$ is well-defined and $\delta_i^2 = 0$.
It satisfies a ``Leibniz formula":
\begin{equation} \label{gen macdonald}
\delta_i (c p_{[v]}) =  \displaystyle (\partial_i c) p_{[v]} + (s_ic) \delta_i p_{[v]}
\end{equation}
where $\partial_i$ is the BGG divided difference operator on $S(\mathfrak{t}^*)$.
The action of $\delta_i$ on the basis of flow-up classes is given by
\begin{equation} \label{div diff formula}
\delta_i p_{[v]} = \left\{ \begin{array}{ll} p_{[s_iv]} & \textup{ if } [s_iv]<_P[v] \textup{ and } \\
0 & \textup{ otherwise}. \end{array} \right.\end{equation}
\end{proposition}

\begin{proof}
The formula for $\delta_i$ is $\C$-linear because the $W$-action is.  Write $p = \sum c_{[v]} p_{[v]}$ in terms of the basis of flow-up classes.  We obtain
\[  \delta_i p  \displaystyle = \frac{\sum_{[v]} (c_{[v]}p_{[v]} - s_i \cdot (c_{[v]}p_{[v]}))}{\alpha_i} = 
 \sum_{[v]} \delta_i (c_{[v]}p_{[v]}).\]
To show $\delta_i p \in H^*_T(G/P)$ it suffices to show $\delta_i(c p_{[v]}) \in H^*_T(G/P)$ for each $c \in S(\mathfrak{t}^*)$ and flow-up class $p_{[v]} \in H^*_T(G/P)$.  Equation \eqref{gen macdonald} follows from
\[  \delta_i (c p_{[v]}) =  \displaystyle \frac{cp_{[v]} - (s_ic)p_{[v]}}{\alpha_i}+ \frac{(s_ic)p_{[v]} - (s_ic)(s_i \cdot p_{[v]})}{\alpha_i}\]
For every $c \in S(\mathfrak{t}^*)$ the BGG divided difference satisfies $\partial_i c \in S(\mathfrak{t}^*)$, and $s_i c \in H^*_T(G/P)$ as well.  By Equation \eqref{gen macdonald} it suffices to show $\delta_i p_{[v]} \in H^*_T(G/P)$.  But $\delta_i p_{[v]} \in H^*_T(G/P)$ by Equation \eqref{div diff formula}, which is a direct restatement of Lemma \ref{formula for action}.  Thus $\delta_i$ is well-defined on $H^*_T(G/P)$.  

The relation $\delta_i^2=0$ is a result of the expression
\[ \frac{p - s_i \cdot p}{\alpha_i} - s_i \cdot \left( \frac{p - s_i \cdot p}{\alpha_i} \right) = \frac{p - s_i \cdot p}{\alpha_i} - \frac{s_i \cdot p - s_i^2 \cdot p}{-\alpha_i}.\]
\end{proof}

Figure \ref{div diff example} gives all nonzero examples of the divided difference operator $\delta_i$ in the equivariant cohomology of the projective plane.  
\begin{figure}[h]
\begin{picture}(350,140)(0,-35)
\put(15,0){$\delta_1$}
\put(95,0){$=$}
\put(195,0){$-$}
\put(125,-30){\line(1,0){125}}
\put(170,-40){$t_1-t_2$}
\put(275,0){$=$}

\put(30,-20){\circle*{5}}
\put(30,20){\circle*{5}}
\put(50,0){\circle*{5}}
\multiput(30,-20)(0,4){10}{\circle*{1}}
\put(30,-20){\line(1,1){20}}
\put(50,-1){\line(-1,1){20}}
\put(50,1){\line(-1,1){20}}

\put(35,-22){$0$}
\put(55,-2){$t_1-t_2$}
\put(35,18){$t_1-t_3$}

\put(130,-20){\circle*{5}}
\put(130,20){\circle*{5}}
\put(150,0){\circle*{5}}
\multiput(130,-20)(0,4){10}{\circle*{1}}
\put(130,-20){\line(1,1){20}}
\put(150,-1){\line(-1,1){20}}
\put(150,1){\line(-1,1){20}}

\put(135,-22){$0$}
\put(155,-2){$t_1-t_2$}
\put(135,18){$t_1-t_3$}

\put(215,-20){\circle*{5}}
\put(215,20){\circle*{5}}
\put(235,0){\circle*{5}}
\multiput(215,-20)(0,4){10}{\circle*{1}}
\put(215,-20){\line(1,1){20}}
\put(235,-1){\line(-1,1){20}}
\put(235,1){\line(-1,1){20}}

\put(220,-22){$t_2 - t_1$}
\put(240,-2){$0$}
\put(220,18){$t_2-t_3$}

\put(300,-20){\circle*{5}}
\put(300,20){\circle*{5}}
\put(320,0){\circle*{5}}
\multiput(300,-20)(0,4){10}{\circle*{1}}
\put(300,-20){\line(1,1){20}}
\put(320,-1){\line(-1,1){20}}
\put(320,1){\line(-1,1){20}}

\put(305,-22){$1$}
\put(325,-2){$1$}
\put(305,18){$1$}

\put(15,80){$\delta_2$}
\put(95,80){$=$}
\put(195,80){$-$}
\put(125,50){\line(1,0){125}}
\put(170,40){$t_2-t_3$}
\put(275,80){$=$}

\put(30,60){\circle*{5}}
\put(30,100){\circle*{5}}
\put(50,80){\circle*{5}}
\multiput(30,60)(0,4){10}{\circle*{1}}
\put(30,60){\line(1,1){20}}
\put(50,79){\line(-1,1){20}}
\put(50,81){\line(-1,1){20}}

\put(35,58){$0$}
\put(55,78){$0$}
\put(35,98){$(t_1-t_3)(t_2-t_3)$}

\put(130,60){\circle*{5}}
\put(130,100){\circle*{5}}
\put(150,80){\circle*{5}}
\multiput(130,60)(0,4){10}{\circle*{1}}
\put(130,60){\line(1,1){20}}
\put(150,79){\line(-1,1){20}}
\put(150,81){\line(-1,1){20}}

\put(135,58){$0$}
\put(155,78){$0$}
\put(135,98){\small $(t_1-t_3)(t_2-t_3)$}

\put(215,60){\circle*{5}}
\put(215,100){\circle*{5}}
\put(235,80){\circle*{5}}
\multiput(215,60)(0,4){10}{\circle*{1}}
\put(215,60){\line(1,1){20}}
\put(235,79){\line(-1,1){20}}
\put(235,81){\line(-1,1){20}}

\put(220,58){$0$}
\put(230,88){\tiny $(t_1-t_2)(t_3-t_2)$}
\put(220,98){$0$}

\put(300,60){\circle*{5}}
\put(300,100){\circle*{5}}
\put(320,80){\circle*{5}}
\multiput(300,60)(0,4){10}{\circle*{1}}
\put(300,60){\line(1,1){20}}
\put(320,79){\line(-1,1){20}}
\put(320,81){\line(-1,1){20}}

\put(305,58){$0$}
\put(325,78){$t_1-t_2$}
\put(305,98){$t_1-t_3$}

\end{picture}
\caption{Divided difference operators on Schubert classes in $H^*_T(\mathbb{P}^2)$} \label{div diff example}
\end{figure}

The next proposition proves that these $\delta_i$ satisfy the braid relations.  The main step uses Equation \eqref{gen macdonald} to reduce the problem to a calculation on the BGG divided difference operators on $H^*(G/B)$.  This key calculation is performed in Lemma \ref{key braid calculation} following the proof of the proposition.  

We use the following three definitions.

\begin{definition}
For each $w \in W$, the set $R(w)$ consists of all reduced words for $w$ in the simple reflections $\{s_1, s_2, \ldots, s_n\}$.
\end{definition}

\begin{definition}
Given a fixed reduced factorization $w=w_1 w_2 \cdots w_k \in W$ and a reduced word $u \leq w$, let $\mathcal{S}_u^w$ be the set of integer sequences $1 \leq b_1 < b_2 < \cdots < b_{\ell(u)} \leq k$ such that $w_{b_1}w_{b_2} \cdots w_{b_{\ell(u)}} \in R(u)$.
\end{definition}

\begin{definition}
Given a reduced factorization $w=w_1 w_2 \cdots w_k \in W$ and an integer sequence ${\bf b} = b_1 b_2 \cdots b_l$ satisfying $1 \leq b_1 < b_2 < \cdots < b_l \leq k$, define $\iota_{w}({\bf b})$ to be the word in $\{s_1, s_2, \ldots, s_n, \partial_1, \partial_2, \ldots, \partial_n\}$ obtained from $w$ by switching the letters that are {\em not} in positions $b_1, b_2, \ldots, b_l$ from $s_i$ to $\partial_i$.
\end{definition}

For instance, in type $A_n$ if $w = s_1s_2s_1$ then $R(w) = \{s_1s_2s_1, s_2s_1s_2\}$.  The set $\mathcal{S}_{s_1}^{s_1s_2s_1} = \{1,3\}$ which gives the words $\iota_{s_1s_2s_1}(1) = s_1 \partial_2 \partial_1$ and $\iota_{s_1s_2s_1}(3) = \partial_1 \partial_2 s_1$.

\begin{proposition} \label{braid relations}
The divided difference operators $\{\delta_i\}$ satisfy the braid relations.
\end{proposition}

\begin{proof}
Suppose $(s_is_j)^m=(s_js_i)^m$ is a braid relation defining $W$ for $i \neq j$.  We will compare $(\delta_i \delta_j)^m (cp_{[v]})$ to $(\delta_i \delta_j)^m (cp_{[v]})$ for an arbitrary $c \in S(\mathfrak{t}^*)$ and flow-up class $p_{[v]} \in H^*_T(G/P)$.   The flow-up classes form an $S(\mathfrak{t}^*)$-module basis of $H^*_T(G/P)$, so additivity of the divided difference operators $\delta_i$ implies the claim.

Note that if $u < (s_is_j)^m$ then $u$ has a unique minimal-length factorization in $W$.  Thus $\delta_u$ can be defined as the sequence of $\delta_i$ and $\delta_j$ corresponding to $u$.  Repeated application of Equation \eqref{gen macdonald} shows that
\[ (\delta_i \delta_j)^m (cp_{[v]}) = (s_is_j)^m (c)   \left((\delta_i \delta_j)^m p_{[v]} \right) + \sum_{\footnotesize \begin{array}{c} u < (s_is_j)^m s.t. \\ \ell_P([uv]) = \\  \ell_P([v]) - \ell(u) \end{array}} (\delta_u p_{[v]}) \left( \sum_{{\bf b} \in \mathcal{S}_u^{(s_is_j)^m}} \iota_{(s_is_j)^m}({\bf b}) c \right)\]
and similarly for $(\delta_j \delta_i)^m (cp_{[v]})$.
The coefficient of $\delta_u p_{[v]}$ is a sum of a sequence of BGG divided difference operators applied to elements of $S(\mathfrak{t}^*)$.  Lemma \ref{key braid calculation} proves 
\[\sum_{{\bf b} \in \mathcal{S}_u^{(s_is_j)^m}} \iota_{(s_is_j)^m}({\bf b}) c = \sum_{{\bf b} \in \mathcal{S}_u^{(s_js_i)^m}} \iota_{(s_js_i)^m}({\bf b}) c\]
for each $u \leq (s_is_j)^m$.  Since 
\[\{u \in W: u < (s_is_j)^m\} = \{u \in W: u < (s_js_i)^m\}\]
we have 
\[(\delta_i \delta_j)^m (cp_{[v]}) = (\delta_j \delta_i)^m (cp_{[v]})\] 
for all  $c \in S(\mathfrak{t}^*)$ if and only if  
\[(\delta_i \delta_j)^m p_{[v]} = (\delta_j \delta_i)^m p_{[v]}.\]
If $\ell_P([(s_is_j)^mv]) = \ell_P([v]) - 2m$ then repeated application of Equation \ref{div diff formula} shows
\[ (\delta_i \delta_j)^m p_{[v]} = p_{[(s_is_j)^mv]} = (\delta_j \delta_i)^m p_{[v]}.\]
If not, the same Equation gives $(\delta_i \delta_j)^m p_{[v]} = 0 = (\delta_j \delta_i)^m p_{[v]}$.  The claim follows.
\end{proof}

\begin{lemma} \label{key braid calculation}
Assume $i\neq j$ and that $(s_is_j)^m= (s_js_i)^m$ is a braid relation in $W$.  Let $u \in W$ be {\em any} reduced word with $u \leq (s_is_j)^m$.  Then
\begin{equation} \label{claim of calc}
\sum_{{\bf{b}} \in \mathcal{S}_u^{(s_is_j)^m}} \iota_{(s_is_j)^m}({\bf b})  = \sum_{{\bf{b}} \in \mathcal{S}_u^{(s_js_i)^m}} \iota_{(s_js_i)^m}({\bf b}).\end{equation}
\end{lemma}

\begin{proof}
The proof is by induction on the length of $u$.  If $u$ has length zero then $\iota_{(s_is_j)^m}(\emptyset) = (\partial_i \partial_j)^m$ and $\iota_{(s_js_i)^m}(\emptyset) = (\partial_j \partial_i)^m$.  BGG divided difference operators satisfy the braid relations so $\iota_{(s_is_j)^m}(\emptyset)=\iota_{(s_js_i)^m}(\emptyset)$. Similarly, when $u=(s_is_j)^m = (s_js_i)^m$ the sets $\mathcal{S}_{(s_is_j)^m}^{(s_is_j)^m}$ and $\mathcal{S}_{(s_js_i)^m}^{(s_js_i)^m}$ both equal  $\{{\bf b} = (1,2,\ldots,2m)\}$ and so $\iota_{(s_is_j)^m}({\bf b})=(\partial_i\partial_j)^m = \iota_{(s_js_i)^m}({\bf b})$.

Assume that if $\ell(u) \leq k$ then Equation \eqref{claim of calc} holds for $u$.  For every pair $f,g \in S(\mathfrak{t}^*)$, repeated application of Lemma \ref{Macdonald lemma} shows that
\[(\partial_i \partial_j)^m (fg) = \sum_{u \leq (s_is_j)^m} (\partial_u g) \left( \sum_{{\bf{b}} \in \mathcal{S}_u^{(s_is_j)^m}} \iota_ {(s_is_j)^m}({\bf b})f  \right)\]
and similarly for $(\partial_j \partial_i)^m(fg)$.
Note that $u \leq (s_is_j)^m$ if and only if $u \leq (s_js_i)^m$.  

If $u < (s_is_j)^m$ then $u$ has a unique factorization into simple reflections.  Choose $u_0 < (s_is_j)^m$ of length $k+1 \neq m$ and let $g$ be any Schubert polynomial $\mathfrak{S}_{u_0}$ satisfying the conditions set out at the beginning of this Section.  In particular $\partial_i \mathfrak{S}_u = 0$ whenever $us_i \geq u$.  It follows that if $\ell(u) \geq k+1$ and $u \neq u_0$ then  $\partial_u \mathfrak{S}_{u_0} = 0$.  The inductive hypothesis proved that if $\ell(u) \leq k$ then 
\[(\partial_u  \mathfrak{S}_{u_0}) \left(  \sum_{{\bf b} \in \mathcal{S}_u^{(s_is_j)^m}} \iota_ {(s_is_j)^m}({\bf b})f  \right) = (\partial_u  \mathfrak{S}_{u_0}) \left(  \sum_{{\bf b} \in \mathcal{S}_u^{(s_js_i)^m}} \iota_ {(s_js_i)^m}({\bf b})f  \right).\]
It follows that 
\begin{equation} \label{temp} 
\begin{array}{l} (\partial_i \partial_j)^m(f \mathfrak{S}_{u_0}) - (\partial_j \partial_i)^m (f \mathfrak{S}_{u_0}) = \\  \hspace{.5in} (\partial_{u_0}  \mathfrak{S}_{u_0})\left( \sum_{{\bf b} \in \mathcal{S}_{u_0}^{(s_is_j)^m}} \iota_ {(s_is_j)^m}({\bf b})f  -  \sum_{{\bf b} \in \mathcal{S}_{u_0}^{(s_js_i)^m}} \iota_ {(s_js_i)^m}({\bf b})f  \right).\end{array}\end{equation}
Recall that $\partial_{u_0}  \mathfrak{S}_{u_0}$ was assumed to be a nonzero integer.  The nil-Coxeter relations for the BGG divided difference operators imply that $(\partial_i\partial_j)^m = (\partial_j \partial_i)^m$ and so Equation \eqref{temp} is zero.  Since $f$ was arbitrary Equation \eqref{claim of calc} holds for $u_0$.  By induction it is true for all $u \leq (s_is_j)^m$.
\end{proof}

Together, these results show $\delta_w$ is well-defined for each $w \in W$. By Proposition \ref{braid relations}, we have  $\delta_{{i_1}} \cdots \delta_{i_k} p_{[u]}= p_{[s_{i_1} \cdots s_{i_k} u]}$ for each flow-up class $p_{[u]} \in H^*_T(G/P)$ with $\ell_P([s_{i_1} \cdots s_{i_k}u] = \ell_P([u]) - k$.    By Proposition \ref{well-defined, nil coxeter}, we have $\delta_{{i_1}} \cdots \delta_{i_k} p_{[u]} = 0$ otherwise.  Hence $\delta_w=\delta_{{i_1}} \cdots \delta_{i_k}$ for any reduced factorization $w = s_{{i_1}} \cdots s_{i_k}$.

\begin{remark}
Divided difference operators are $\C$-module homomorphisms defined on $H^*_T(G/P)$.  As a $\C$-module, the ordinary cohomology $H^*(G/P)$ is the $\C$-span of the flow-up classes.  (The ring structure of $H^*(G/P)$ is obtained from the ring isomorphism 
\[H^*(G/P) \cong \frac{H^*_T(G/P)}{\ideal{t_1, t_2, \ldots, t_n} H^*_T(G/P)}\] 
\cite[Equation (1.2.4)]{GKM}.)  This means that the divided difference operators $\delta_i$ defined on $H^*_T(G/P)$ also apply to ordinary cohomology, viewed as a submodule of $H^*_T(G/P)$.
\end{remark}

\begin{remark}
Kostant-Kumar defined a different divided difference operator on $H^*_T(G/B)$ using the GKM presentation \cite{KK}.  We denote this divided difference operator $\partial_i^{KK}$ because Arabia proved that it induces the BGG divided difference operator on ordinary cohomology \cite{A}.  If $p \in H^*_T(G/B)$ is a class in the GKM presentation then for each $v \in W$
\[(\partial_i^{KK}p)(v) = \frac{p(v)-p(vs_i)}{-v(\alpha_i)}.\]
We remark that Kostant-Kumar's divided difference operator is not well-defined on $H^*_T(G/P)$.  For a more detailed treatment see \cite{T2}.  
\end{remark}

\section{Generating Schubert classes} \label{generating class section}

A classical result for flag varieties states that successive application of BGG divided difference operators on the Schubert class $p_{w_0}$ of the longest element $w_0 \in W$ generates all Schubert classes \cite{BGG}.  We prove this for Grassmannians with the divided difference operators $\delta_i$.

\begin{theorem} \label{gen all classes}
Each Schubert class $p_{[w]}$ in $H^*(G/P)$ and in $H^*_T(G/P)$ can be obtained as $\delta_v p_{[w_0]}$ for some $v \in W$.
\end{theorem}

\begin{proof}
Recall that $W/W_J$ has an order-reversing bijection $\varphi: W/W_J \rightarrow W/W_J$ with $[u] \leq_P [w]$ if and only if $\varphi([w]) \leq_P \varphi([u])$.  If $w_0$ is the longest element of $W$ then $\varphi$ can be defined by $\varphi([u]) = [w_0u]$.  (Among others, Bj\"{o}rner-Brenti prove this \cite[Proposition 2.5.4]{BB}.)  

The shortest element of $W/W_J$ is the identity $[e]$ and is sent to the longest element $\varphi([e]) = [w_0]$.  The edge $[u] \mapsto [s_{\alpha}u]$ maps to the edge $[w_0s_{\alpha}u] \mapsto [w_0u]$, which is labeled $w_0(\alpha)$.  If $s_i$ is a simple reflection then $w_0s_iw_0^{-1}$ is also simple.

Each element $[w_0^{-1}u] \in W/W_J$ has a reduced factorization into simple reflections, say $[w_0^{-1}u] = s_{i_1} \cdots s_{i_k}[e]$ with $\ell_P([w_0^{-1}u]) = k$.  Suppose that $w_0s_{i_j}w_0^{-1} = s_{i_j'}$ for each $j$.  Applying $\varphi$ shows that $[u] = s_{i_1'} \cdots s_{i_k'} [w_0]$ and $\ell_P([w_0]) = \ell_P([u])+k$.  Equation \eqref{div diff formula} proves $\delta_{i_1'} \cdots \delta_{i_k'} p_{[w_0]} = p_{[u]}$.
\end{proof}

\section{A generalized Billey formula for flow-up classes} \label{billey section}

Billey gives a formula for the localizations of equivariant Schubert classes in flag varieties \cite{Bi2}.  We use Corollary \ref{basic formula} to generalize Billey's formula to arbitrary partial flag varieties. The new formula is similar to the original: to compute the localization $p_{[w]}([v])$, use minimal-length representatives $w \in [w]$ and $v \in [v]$ and compute $p_w(v)$ in the equivariant cohomology of the full flag variety.  In fact, we will show that $v$ may be chosen arbitrarily from the coset $[v]$, though $w$ must be minimal-length in $[w]$.  Billey's proof used the divided difference formula of Kostant-Kumar \cite{KK}, which is not well-defined on $H^*_T(G/P)$.

\begin{theorem} \label{Billey generalization}
Let $[v],[w]$ be two elements of $W/W_J$ and assume that $v,w$ are minimal in $[v],[w]$. Fix a reduced word for $v$, say $v=b_1\cdots b_l$.  Let $R(w)$ denote the set of reduced words for $w$ and for each $i$ let $r_b(i)=b_1b_2 \cdots b_{i-1}(\alpha_i)$.  Then
\[p_{[w]}([v]) = p_w(v) = \sum_{b_{i_1}\cdots b_{i_k} \in R(w)} r_b(i_1) r_b(i_2)\cdots r_b(i_k)\]
where the sum is over $1 \leq i_1 < i_2 < \cdots < i_k \leq l$ with $b_{i_1}\cdots b_{i_k} \in R(w)$.
\end{theorem}

\begin{proof}
The proof is by induction on the length of $w$.  If $[w]$ is the maximal element in $W/W_J$ then $p_{[w]} ([v]) = 0$ unless $[v]=[w]$, in which case 
\[p_{[w]}([w])  = \prod_{\alpha \in \Inv{[w]}^P} \alpha = r_b(1)r_b(2)r_b(3) \cdots r_b(l)\]
by definition of flow-up classes and the exchange condition \cite[Corollary 1.4.4]{BB}. 

Assume by induction that the claim holds for all elements of $W/W_J$ with length at least $k+1$.  Let $[w] \in W/W_J$ have $\ell_P([w])=k$ and take $w$ to be in $W^J$, namely $w$ is the minimal-length representative of the coset $[w]$.  Since $[w]$ is not the longest element of $W/W_J$, there is a simple reflection $s_i$ with $[s_iw] >_P [w]$.  In particular, no reduced word for $w$ begins with $s_i$.

We now compute $p_{[s_iw]}([v])$ for each $[v] \in W/W_J$.  We will then use the formula from Corollary \ref{basic formula} to deduce $p_{[w]}([v])$. The proof is divided into two cases: when $[s_iv] <_P [v]$ and when $[s_iv] \geq_P [v]$.  Assume that $v$ has been chosen in $W^J$, namely that $v$ is minimal-length in the coset $[v]$.

If $[v] >_P [s_iv]$ then $s_iv$ is the minimal-length element of $[s_iv]$ by Lemma \ref{simple edges}. In this case we may choose $b_1=s_i$ in the reduced word for $v$.  
By induction we know
\[p_{[s_iw]}([v])= \sum_{b_{i_0}\cdots b_{i_{k}} \in R(s_iw)} r_b(i_0)\cdots r_b(i_k)\]
which decomposes into the sum
\[p_{[s_iw]}([v])= \alpha_i \sum_{b_{i_1}\cdots b_{i_k} \in R(w)} r_b(i_1)\cdots r_b(i_k)
	+  \sum_{b_{i_0}\cdots b_{i_k} \in R(s_iw), i_0 \neq 1} r_b(i_0)\cdots r_b(i_k).\]
Similarly we conclude
\[p_{[s_iw]}([s_iv])= s_i \left( \sum_{b_{i_0}\cdots b_{i_k} \in R(s_iw), i_0 \neq 1} r_b(i_0)\cdots r_b(i_k)\right). \]
Consequently
\[p_{[s_iw]}([v]) - s_ip_{[s_iw]}([s_iv]) =  \alpha_i \sum_{b_{i_1}\cdots b_{i_k} \in R(w)} r_b(i_1)\cdots r_b(i_k).\]
Corollary \ref{basic formula} states that this quantity equals $\alpha_i p_{[w]}([v])$, which proves the claim.

If $[v] \leq_P [s_iv]$ then the element $s_iv>v$ again by Lemma \ref{simple edges}.  Choose a reduced factorization $b_0b_1 \cdots b_{\ell(v)}$ for $s_iv$ so that both $b_0=s_i$ and $b=b_1 \cdots b_{\ell(v)}$ is a reduced factorization for $v$.  The inductive hypothesis shows that
\[\begin{array}{l} \displaystyle
p_{[s_iw]}([s_iv]) = \sum_{b_{i_0}\cdots b_{i_k} \in R(s_iw)} r_{b_0b}(i_0)\cdots r_{b_0b}(i_k) \\
 \displaystyle \hspace{.5in}= \alpha_i s_i \left( \sum_{b_{i_1}\cdots b_{i_k} \in R(w), i_1 \geq 1} r_b(i_1)\cdots r_b(i_k) \right) +  s_i \left( \sum_{b_{i_0}\cdots b_{i_k} \in R(s_iw), i_0 \geq 1} r_{b}(i_0)\cdots r_{b}(i_k) \right) \\ 
 \displaystyle	\hspace{.5in}= \alpha_i s_i \left( \sum_{b_{i_1}\cdots b_{i_k} \in R(w), i_1 \geq 1} r_b(i_1)\cdots r_b(i_k)	\right) +  s_ip_{[s_iw]}([v]).\end{array}\]
Again we conclude
\[\frac{p_{[s_iw]}([v])-s_ip_{[s_iw]}([s_iv])}{\alpha_i} =  \sum_{b_{i_1}\cdots b_{i_k} \in R(w)} r_b(i_1)\cdots r_b(i_k) = p_{[w]}([v]).\]
\end{proof}

Billey proved that this expression is independent of the word chosen for $v$  \cite{Bi2}.  It is not independent of the word chosen for $[w]$: for instance, both $s_1s_2s_1$ and $s_1s_2$ represent the class $[s_1s_2]$ in $G(1,3)$, but $p_{s_1s_2s_1}$ is a class of degree three while $p_{s_1s_2}$ has degree two.  While $w \in [w]$ must be minimal-length, we will show that the localizations of $p_w$ in $H^*_T(G/B)$ agree on all the different representatives in $[v]$.

To prove this, we use Kostant-Kumar's divided difference operator on the GKM presentation of $H^*_T(G/B)$ (see \cite{KK} or \cite{Bi2}).  Kostant-Kumar showed that for each $w \in W$ the image $\partial_i^{KK} p_{w} = 0$ if $ws_i>w$.  The next lemma follows immediately.

\begin{lemma}
For each $w \in W$ and flow-up class $p_w \in H^*_T(G/B)$, if $s_i$ is a simple reflection such that $ws_i > w$ then $p_w(v) = p_w(vs_i)$ for all $v \in W$. 
\end{lemma}

\begin{proof}
If $ws_i > w$ then $\partial_i^{KK} p_w = 0$, so $p_w(v) = p_w(vs_i)$ for all $v \in W$.
\end{proof}

\begin{corollary}
For each $w \in W^J$, flow-up class $p_w \in H^*_T(G/B)$, $v \in W$, and $u \in W_J$, 
\[p_{w}(v) = p_{w}(vu).\]
\end{corollary}

\begin{proof}
If $w \in W^J$ then $ws_i > w$ for each $s_i \in W_J$. (This is a classical result proven by, e.g., Bj\"{o}rner-Brenti \cite[Lemma 2.4.3]{BB}.)  The previous lemma showed $p_w(v)= p_w(vs_i)$ for all $v \in W$ and $s_i \in W_J$.  The subgroup $W_J$ is generated by $\{s_i: s_i \in W_J\}$, so $p_w(v)=p_w(vu)$ for all $u \in W_J$.
\end{proof}

Figure \ref{localization table} gives the localizations of the class $p_{[s_2]}$ in $H^*_T(G(2,4))$, the Grassmannian of $2$-planes in $\C^4$.  The roots $\alpha_i$ are written as $t_i - t_{i+1}$.

\begin{figure}[h]
\begin{tabular}{|c|c|c|c|c|c|c|}
\cline{1-7} minimal $w \in [w]$ & $e$ & $s_2$ & $s_3s_2$ & $s_1s_2$ & $s_1s_3s_2$ & $s_2s_1s_3s_2$ \\
\cline{1-7} $p_{[s_2]}([w])$ & 0 & $t_2 - t_3$ & $t_2-t_4$ & $t_1 - t_3$ & $t_1 - t_4$ & $t_1 - t_4 + t_2 - t_3$ \\ \hline
\end{tabular}
\caption{} \label{localization table}
\end{figure}

Billey's formula also implies the following.

\begin{corollary}
The map $\varphi: H^*_T(G/B) \rightarrow H^*_T(G/P)$ that sends the class $p = (p(v))_{v \in W}$ to $\varphi(p) = (p(v))_{v \in W^J}$ restricts to an $S(\mathfrak{t}^*)$-algebra isomorphism on the $S(\mathfrak{t}^*)$-span of $\{ p_w: w \in W^J \}$.
\end{corollary}

\begin{proof}
Billey's formula proves that for each $w \in W^J$, the class $\varphi(p_w)$ satisfies all the edge relations specified by the GKM conditions for $G/P$.  It also shows that $\varphi(p_w) = p_{[w]}$ is the flow-up class corresponding to $[w] \in W/W_J$.  This confirms the image of the map $\varphi: \langle p_w: w \in W^J \rangle \rightarrow H^*_T(G/P)$.  The algebra structure in the GKM presentation for both $H^*_T(G/B)$ and $H^*_T(G/P)$ is defined coordinate-wise, so $\varphi$ is a $S(\mathfrak{t}^*)$-algebra isomorphism by definition.
\end{proof}

This isomorphism is the inverse of the natural geometric map $H^*_T(G/P) \hookrightarrow H^*_T(G/B)$ induced from the projection $G/B \rightarrow G/P$ (e.g. \cite[Corollary 11.3.14]{Ku}). 

\section{Schubert polynomials for Grassmannians} \label{schubert poly section}

Most of this paper has focused on a geometric construction of the (equivariant) cohomology of Grassmannians.  This section will focus on an algebraic construction of the cohomology ring.  In it, we propose an explicit collection of Schubert polynomials for the (ordinary) cohomology of Grassmannians.  These Schubert polynomials are defined using the intrinsic geometry of Schubert varieties, as the average of localizations of a corresponding equivariant Schubert class.  Lascoux also defined Grassmannian double Schubert polynomials by restricting the ordinary double Schubert polynomials to Grassmannian permutations \cite{L}.  Our approach 1) clearly demonstrates the $W_J$-invariance of the Grassmannian Schubert polynomials, and 2) gives an explicit closed formula for the Grassmannian Schubert polynomials.

First we review the algebraic construction of the cohomology of the flag variety.  Recall that the Weyl group $W$ acts on the symmetric algebra $S(\mathfrak{t}^*)$, and let $I$ be the ideal in $S(\mathfrak{t}^*)$ generated by $W$-invariant elements $f \in S(\mathfrak{t}^*)$ that satisfy $f(0)=0$.  Borel and Atiyah-Hirzebruch proved that there is a $W$-equivariant ring isomorphism $\alpha: S(\mathfrak{t}^*)/I \rightarrow H^*(G/B)$ \cite{Bo}, \cite{AH}.  Seminal work of Bernstein-Gelfand-Gelfand and Demazure identified explicitly a family of so-called Schubert polynomials $\mathfrak{S}_w \in S(\mathfrak{t}^*)$ that represent the Schubert classes \cite{BGG}, \cite{D}.  

The choice of polynomial representatives of Schubert classes in $S(\mathfrak{t}^*)/I$ is not unique.  Combinatorists have discussed passionately which properties of Schubert polynomials are essential.  We use a modified version of Fomin-Kirillov's \cite{FK}: 
\begin{enumerate}
\item the polynomial $\mathfrak{S}_w$ is homogeneous of degree $\ell(w)$;
\item the polynomial $\mathfrak{S}_w$ is positive in the $\alpha_i$, namely each coefficient is a nonnegative integer;
\item the $\mathfrak{S}_w$ respect BGG divided difference operators in the sense that
\[\partial_i \mathfrak{S}_w = \left\{ \begin{array}{ll} \mathfrak{S}_{ws_i} & \textup{ if }ws_i < w \textup{ and} \\ 0 & \textup{ otherwise;} \end{array} \right.\]
\item and the product of Schubert polynomials gives the structure constants $c_{uv}^w$ in $H^*(G/B)$ up to the ideal $I$ via the expression
\[ \mathfrak{S}_u \mathfrak{S}_v = \sum c_{uv}^w \mathfrak{S}_w \mod I.\]
\end{enumerate}  
(In fact Fomin-Kirillov require positivity in a different set of variables.)

Bernstein-Gelfand-Gelfand and Demazure defined Schubert polynomials by choosing a top-degree Schubert polynomial and then using divided difference operators (and Property (3) in the previous list) to define the others \cite{BGG}, \cite{D}.  
For brevity, we refer to these polynomials as the BGG Schubert polynomials.  The BGG Schubert polynomials satisfy Properties (1)--(4).   

We will directly construct Schubert polynomials for $H^*(G/P)$.  We will use another classical result of Bernstein-Gelfand-Gelfand: that $H^*(G/P)$ is isomorphic to the $W_J$-invariant subring of $S(\mathfrak{t}^*)/I$ \cite[Theorem 5.5]{BGG}.  Our Schubert polynomials will also satisfy Properties (1)--(4).  Note that the ideal $I$ that appears in Property (4) is determined by $G$ and is independent of the parabolic $P$.

Our first proposition, due to C.~Cadman, gives a formula for the BGG Schubert polynomials in terms of the flow-up classes.

\begin{proposition} (Cadman)
The map $\kappa: H^*_T(G/B) \rightarrow S(\mathfrak{t}^*)/I$ defined by
\[ \kappa(p) = \frac{\sum_{w\in W} p(w)}{|W|}\]
is a degree-preserving $S(\mathfrak{t}^*)$-module homomorphism that restricts to a $\C$-module isomorphism on the submodule $H^*(G/B)$.  For each $w\in W$ the map sends $\kappa(p_w) = \mathfrak{S}_{w^{-1}}$.  The map $\kappa$ is $W$-equivariant, namely $\kappa (w \cdot p) = w(\kappa(p))$ for each $w \in W$ and $p \in H^*_T(G/B)$.  Moreover, the map $\kappa$ commutes with divided difference operators in the sense that $\kappa \circ \delta_i = \partial_i \circ \kappa$ for each $i$.
\end{proposition}

\begin{proof}
The map $\kappa$ is an $S(\mathfrak{t}^*)$-module homomorphism because for each $w \in W$, the localization map $H^*_T(G/B) \rightarrow H^*_T(pt)$ given by $p \mapsto p(w)$ is an $S(\mathfrak{t}^*)$-module homomorphism.  (It preserves degree by definition.)  The submodule $H^*(G/B)$ is the $\C$-span in $H^*_T(G/B)$ of the equivariant Schubert classes $\{p_w: w \in W\}$, so $\kappa$ restricts to a $\C$-module homomorphism on $H^*(G/B)$.  For each $p \in H^*_T(G/B)$ and $w \in W$ we have
\[ \frac{\sum_{v \in W} wp(w^{-1}v)}{|W|} = w \left( \frac{\sum_{v \in W} p(w^{-1}v)}{|W|} \right) = w \left( \frac{\sum_{v\in W} p(v)}{|W|} \right)\]
which proves that $\kappa(w \cdot p) = w(\kappa(p))$.  In addition, for each $p \in H^*_T(G/B)$ 
\[ \frac{\sum_{v \in W} (\delta_i p)(v)}{|W|} = \frac{\sum_{v \in W} \left( p(v) - s_ip(s_iv) \right) }{|W| \alpha_i} = \partial_i \left( \frac{ \sum_{v \in W} p(v)}{|W|} \right).\]

We prove that $\kappa(p_w) = \mathfrak{S}_{w^{-1}}$ by induction on the length of $w$.  The BGG Schubert polynomial for the longest permutation is defined to be $\mathfrak{S}_{w_0} = \displaystyle \frac{\prod_{\alpha \in \Delta^+} \alpha}{|W|}$.  The Schubert class $p_{w_0}$ has localizations $p_{w_0}(v)=0$ if $v \neq w_0$ and $p_{w_0}(w_0) = \prod_{\alpha \in \Delta^+} \alpha$.  So the claim holds for the longest permutation.

Assume the claim holds for $w$ of length $k$.  If $s_i$ is a simple reflection with $s_iw < w$ then $\delta_i p_w = p_{s_iw}$.  Since $\kappa$ commutes with $\delta_i$, we conclude $\kappa(p_{s_iw}) = \partial_i \mathfrak{S}_{w^{-1}} = \mathfrak{S}_{w^{-1}s_i}$.  The claim holds by induction.

The $\C$-module homomorphism $\kappa: H^*(G/B) \rightarrow S(\mathfrak{t}^*)/I$ bijectively maps a free basis in the domain to a free basis in the range, so it is an isomorphism.
\end{proof}

Bernstein-Gelfand-Gelfand explicitly identified the images of their Schubert polynomials under the map of Borel and Atiyah-Hirzebruch; with our conventions, they found \cite[Section I, Theorem (ii)]{BGG}
\begin{equation} \label{BGG map}
\mathfrak{S}_w I \stackrel{\alpha}{\mapsto} p_w.
\end{equation}
The maps $\kappa$ and $\alpha$ both use the same $W$-action on $S(\mathfrak{t}^*)/I$, but {\em different} $W$-actions on $H^*_T(G/B)$.  The $W$-action on $H^*(G/B)$ that is used in \cite{Bo}, \cite{AH}, \cite{BGG} is very complicated, making it difficult to identify $W_J$-invariant polynomials.  (This $W$-action is discussed in the GKM setting in \cite{KK}.  In fact the graded $\mathbb{C}$-module $H^*(G/B)$ with the Kostant-Kumar $W$-action is the regular representation \cite{T2}.)

The $W$-action on $H^*_T(G/B)$ that we use in this paper is much simpler.  Bernstein-Gelfand-Gelfand showed that $H^*(G/P)$ is isomorphic to the $W_J$-stable elements of $S(\mathfrak{t}^*)/I$ \cite[Theorem 5.5]{BGG}.  We now construct Schubert polynomials for Grassmannians by restricting the map $\kappa$ to $W_J$-stable classes in $H^*_T(G/B)$.  To do that, we use minimal representatives of {\em left} $W_J$ cosets.

\begin{theorem} \label{schubert poly theorem}
The ring $H^*(G/P)$ is isomorphic to the subring of $S(\mathfrak{t}^*)/I$ spanned by $\{\mathfrak{S}_w: w \in W^J\}$, where each $\mathfrak{S}_w$ is the image $\kappa(p_{w^{-1}})$.  The Grassmannian Schubert polynomial $\mathfrak{S}_w$ is
\begin{enumerate}
\item a  homogeneous polynomial of degree $\ell_P([w])$;
\item positive in the simple roots $\alpha_i$; 
\item and gives the structure constants $c_{uv}^w$ in $H^*(G/P)$ up to the ideal $I$ via the expression
\[ \mathfrak{S}_u \mathfrak{S}_v = \sum c_{uv}^w \mathfrak{S}_w \mod I.\]
\end{enumerate}
\end{theorem}

\begin{proof} 
The map $\kappa$ is $W$-equivariant, so it suffices to identify $W_J$-invariant elements of $H^*_T(G/B)$.  The equivariant Schubert class $p_w \in H^*_T(G/B)$ is $W_J$-invariant if and only if $s_i \cdot p_w = p_w$ for each $i \in J$, since the set $\{s_i: i \in J\}$ generates $W_J$.  By Lemma \ref{formula for action}, this is equivalent to the condition that $s_i w > w$ for each $i \in J$, namely $w$ is a minimal representative of its {\em left} $W_J$ coset.  In other words, the set $\{p_{w^{-1}}: w \in W^J\}$ is precisely the collection of $W_J$-invariant Schubert classes.  

Applying $\kappa$, we see that $\{\mathfrak{S}_w: w \in W^J\}$ is a $W_J$-invariant set of polynomials.  Each polynomial is positive in the simple roots $\alpha_i$ and homogeneous of the appropriate degree because the flow-up classes $p_w$ are.  BGG Schubert polynomials are linearly independent in $S(\mathfrak{t}^*)/I$ since they form a basis for $H^*(G/B)$, so $\{\mathfrak{S}_w: w \in W^J\}$ is linearly independent in the $W_J$-invariant submodule of $S(\mathfrak{t}^*)/I$.  The map $\kappa$ preserves degree so each $\mathfrak{S}_w$ is a homogeneous polynomial of degree $\ell(w) = \ell_P([w])$.  Comparing degrees, we conclude that $\{\mathfrak{S}_w: w \in W^J\}$ is a basis for the $W_J$-invariant submodule of $S(\mathfrak{t}^*)/I$.  

To see that these BGG Schubert polynomials satisfy the condition that
\[\mathfrak{S}_u \mathfrak{S}_v = \sum c_{uv}^w \mathfrak{S}_w \mod I\]
where $c_{uv}^w$ are the structure constants of the Schubert classes $p_{[u]}$, $p_{[v]}$, and $p_{[w]}$ in $H^*(G/P)$, consider the composition $\alpha \circ \kappa$:
\[\begin{array}{rcl} \langle p_{w^{-1}}: w \in W^J \rangle &\stackrel{\kappa}{\rightarrow}& \langle \mathfrak{S}_w: w \in W^J \rangle \\
& \stackrel{\alpha}{\rightarrow}& \langle p_w: w \in W^J \rangle. \end{array}\]
The map $\alpha$ is a $\C$-algebra isomorphism so the structure constants in $\mathfrak{S}_u \mathfrak{S}_v = \sum c_{uv}^w \mathfrak{S}_w \mod I$ agree with those in the expression $p_u p_v = \sum c_{uv}^w p_w$ for the full flag variety.  By Theorem \ref{Billey generalization}, if $w \in W^J$ then restricting the class $p_w \in H^*_T(G/B)$ to the fixed points in $W^J$ gives the class $p_{[w]} \in H^*_T(G/P)$.  Hence the coefficient $c_{[u],[v]}^{[w]}$ in the expression $p_{[u]}p_{[v]} = \sum c_{[u],[v]}^{[w]} p_{[w]}$ for the partial flag variety $G/P$ satisfies $c_{[u],[v]}^{[w]} = c_{u,v}^w$ for $u,v,w \in W^J$.
\end{proof}

We stress that these Grassmannian Schubert polynomials are completely general, and in particular depend neither on the Lie type of the group nor on particular properties of the parabolic $P$.

Note that these Grassmannian Schubert polynomials have no divided difference operators, since $\partial_i$ disrupts right cosets in $W/W_J$.  Kostant-Kumar's divided difference operator \cite{KK} is well-defined on the classes $\{p_{w^{-1}}: w \in W^J\}$, but it does not commute with the map $\kappa$.  

\subsection{Calculating Schubert polynomials} The most serious objection to Schubert polynomials without divided difference operators is: how could one compute them?  The Billey formula in fact gives a closed formula for the Schubert polynomials we propose.  (In fact, Billey's formula gives the BGG Schubert polynomials as well, though we only give examples for Grassmannians.)    Lemma \ref{formula for action} shows that if $w \in W^J$ then $\kappa(p_{w})$ can be computed either by averaging the entire class in $G/B$ or by averaging the class in $G/P$ and then summing over its $W_J$-orbit, namely 
\[\kappa(p_{w}) = \frac{\sum_{v \in W_J} v \left(\sum_{u \in W^J} p_w(v)\right)}{|W|}.\] 
(This formula does not reduce to scalar multiplication, since $W_J$ can act nontrivially on the polynomial in the parentheses.)  The Schubert polynomials in $\mathbb{P}^2$ are $1$, $\frac{1}{3}(2t_1-t_2-t_3)$, and $\frac{1}{3}(t_1-t_2)(t_1-t_3)$.  The reader may confirm that these polynomials are invariant under the subgroup generated by $s_2$.  Note that they are positive not in the $t_i$ but in the simple roots $\alpha_1 = t_1-t_2$ and $\alpha_2= t_2-t_3$.  Table \ref{table of Schubert polynomials} 
\begin{figure}[h]
\[\begin{tabular}{|c|c|}
\cline{1-2} $w \in W^J$ & $\mathfrak{S}_w$ \\
\cline{1-2} & \\
$s_2s_1s_3s_2$ & \makebox{$\frac{1}{6} \alpha_2(\alpha_1+\alpha_2)(\alpha_2+\alpha_3)(\alpha_1+\alpha_2+\alpha_3)$} \\
& \\
\cline{1-2} & \\
 $s_1s_3s_2$ & $\frac{\alpha_2(\alpha_1+\alpha_2)(\alpha_1+\alpha_2+\alpha_3)+\alpha_2(\alpha_2+\alpha_3)(\alpha_1+\alpha_2+\alpha_3)+(\alpha_1+\alpha_2)(\alpha_2+\alpha_3)(\alpha_1+\alpha_2+\alpha_3)+\alpha_2(\alpha_1+\alpha_2)(\alpha_2+\alpha_3)}{24}$ \\
  &  \\
\cline{1-2} & \\
$s_1s_2$ & $ \frac{1}{12} \left(\alpha_2(\alpha_1+\alpha_2) + (\alpha_2+\alpha_3)(\alpha_1+\alpha_2+\alpha_3) + 2 \alpha_3^2\right)$ \\
& \\
\cline{1-2} & \\
$s_3s_2$ & $\frac{1}{12} \left(\alpha_2(\alpha_2+\alpha_3) + (\alpha_1+\alpha_2)(\alpha_1+\alpha_2+\alpha_3)+ 2 \alpha_1^2\right)$ \\
& \\
\cline{1-2} & \\
$s_2$ & $\frac{1}{2} (\alpha_1 + 2\alpha_2 + \alpha_3)$ \\
& \\
\cline{1-2} $e$ & $1$ \\
\hline \end{tabular}\]
\caption{Schubert polynomials for $G(2,4)$} \label{table of Schubert polynomials}
\end{figure}
shows the Schubert polynomials for $G(2,4)$.  The Billey formula is particularly useful for maximal Grassmannians of classical type, where the moment graph has explicit combinatorial descriptions (such as in Section \ref{examples}).  

\end{document}